\documentclass[11pt,reqno]{amsart}
\usepackage[a4paper,margin=25mm]{geometry}
\usepackage{amsmath,amssymb,mathtools}
\usepackage[colorlinks=true,linkcolor=blue,citecolor=blue,urlcolor=blue]{hyperref}
\allowdisplaybreaks[2]
\newtheorem{theorem}{Theorem}[section]
\newtheorem{lemma}[theorem]{Lemma}
\newtheorem{proposition}[theorem]{Proposition}

\theoremstyle{remark}

\newcommand{\R}{\mathbb R}
\newcommand{\Sph}{S^{n-1}}
\newcommand{\dd}{\,d}
\newcommand{\diver}{\operatorname{div}}
\newcommand{\Lip}{\operatorname{Lip}}
\newcommand{\supp}{\operatorname{supp}}
\newcommand{\headnorm}{\|H_+\|_{L^\theta(\Sph)}}
\newcommand{\forcenorm}{\|f\|_{L^\infty(\Sph)}}
\newcommand{\energyintegral}{\int_{\Sph}\bigl(|\nabla_{\Sph}u|^2+(n-3)|u|^2\bigr)\dd\sigma}
\newcommand{\growthsup}[1][3+1/n-n]{\sup_{\substack{\omega\in\Sph\\0<\rho\le\rho_0}}\rho^{#1}\int_{B^{\Sph}_\rho(\omega)}|u|^2\dd\sigma}
\title[Homogeneous stationary Navier--Stokes solutions]{Homogeneous solutions of the stationary Navier--Stokes equations in arbitrary dimension}
\author{Qi Ma}
\address{Department of Mathematics, Rutgers University--New Brunswick, Hill Center, 110 Frelinghuysen Road, Piscataway, NJ 08854, USA}
\email{qi.ma@rutgers.edu}
\thanks{The author is partially supported by NSF Grant DMS-2247410.}
\date{October 7, 2026}
\hypersetup{pdftitle={Homogeneous solutions of the stationary Navier-Stokes equations in arbitrary dimension},pdfauthor={Qi Ma}}
\makeatletter
\renewcommand{\@settitle}{%
  \begingroup
  \raggedright\normalfont\Large\bfseries
  \noindent\@title\par
  \endgroup
  \vskip1.5em
}
\renewcommand{\@setauthors}{%
  \begingroup
  \raggedright\normalfont\normalsize
  \author@andify\authors
  \noindent{\bfseries\authors}\textsuperscript{*}\par
  \vskip0.6em
  \small
  \def\author##1{}%
  \def\address##1##2{\noindent##2\par}%
  \def\email##1##2{\noindent E-mail: \href{mailto:##2}{##2}\par}%
  \addresses
  \endgroup
}
\renewcommand{\@setthanks}{%
  \begingroup
  \def\thanks##1{\par\noindent\textsuperscript{*}##1\par}%
  \thankses
  \endgroup
}
\let\@setaddresses\relax
\makeatother
\begin{document}
\begin{abstract}
For every integer $n\ge4$, we prove the existence of a $(-1)$-homogeneous solution of the stationary incompressible Navier--Stokes equations on $\R^n\setminus\{0\}$ with an arbitrary locally Lipschitz $(-3)$-homogeneous force. The pressure is $(-2)$-homogeneous, and the solution is regular away from the origin. This extends the homogeneous existence theorem of Bang, Gui, Liu, Wang, and Xie beyond dimension sixteen. The principal estimate combines a subcritical localized pressure identity with an elementary dyadic estimate for an auxiliary tangential vector field in a suitable finite Lebesgue space. Testing the Bernoulli equation at the homogeneous cancellation exponent then controls its positive part in $L^{\frac{(n-2)(n-1)}{2(n-3)}}(S^{n-1})$. The resulting estimate closes with the energy inequality in every fixed dimension. We also establish uniqueness throughout the classical homogeneous class for sufficiently small force. Using the same ideas, we also give a simplified proof of the known existence theorem for bounded compactly supported forces on $\R^n$, $n\ge5$. A general finite-exponent estimate for a scalar equation with divergence-free drift, together with a Euclidean version of the dyadic kernel estimate, bounds the positive Bernoulli function without BMO estimates or an iteration to obtain its supremum norm.
\end{abstract}
\maketitle

\section{Introduction and main results}\label{sec:introduction}

We consider the stationary incompressible Navier--Stokes equations
\begin{equation*}
 -\Delta u+(u\cdot\nabla)u+\nabla p=f,
 \qquad \diver u=0
 \quad\text{in }\R^n\setminus\{0\}.
\end{equation*}
The system is invariant under the scaling
\[
 u(x)\mapsto \lambda u(\lambda x),\qquad
 p(x)\mapsto \lambda^2p(\lambda x),\qquad
 f(x)\mapsto \lambda^3f(\lambda x),\qquad \lambda>0.
\]
A function $g$ is homogeneous of degree $\alpha$ if $g(\lambda x)=\lambda^\alpha g(x)$ for every $\lambda>0$ and $x\in\R^n\setminus\{0\}$. This scaling leads to the question of whether every $(-3)$-homogeneous force admits a $(-1)$-homogeneous stationary solution.

We first recall some developments on the existence and regularity of stationary solutions without a homogeneity assumption. Struwe \cite{Struwe} established a priori estimates and proved the existence of regular solutions on $\R^5$ and on the five-dimensional torus for suitable regular forcing. Frehse and R\r{u}\v{z}i\v{c}ka obtained regular periodic solutions for $5\le n\le15$ \cite{FR95Periodic,FR95} and solutions regular in the interior of bounded six-dimensional domains with homogeneous Dirichlet boundary conditions \cite{FR96Six}. Their weighted estimates and subsequent regularity criteria are developed further in \cite{FR96Weighted,FR98Criterion,FR98Survey}.

For the whole-space problem, Li and Yang \cite{LY} proved in 2022 the existence of regular solutions on $\R^n$, $5\le n\le15$, for arbitrary bounded compactly supported forces, with decay at infinity. Their weighted pressure estimate, local regularity criterion, and construction by Leray--Schauder degree theory are used in the present paper.

We next recall work on homogeneous solutions. \v{S}ver\'ak \cite{Sverak} showed that every nonzero smooth $(-1)$-homogeneous solution of the unforced stationary Navier--Stokes equations on $\R^3\setminus\{0\}$ is a Landau solution. Li Li, YanYan Li, and Xukai Yan studied axisymmetric $(-1)$-homogeneous solutions in dimension three with isolated singularities on the unit sphere, beginning with the case of one singularity in \cite{LLY1}. Their 2018--2019 series \cite{LLY1,LLY2,LLY3} includes classification results for solutions without swirl and existence results for solutions with nonzero swirl.

For the forced homogeneous problem, Shi \cite{Shi} proved in 2018 the existence of $(-1)$-homogeneous weak solutions in dimension $n=4$ for arbitrary $(-3)$-homogeneous forces in $L^{\frac{4}{3},\infty}(\R^4)$, without a smallness assumption.

Bang, Gui, Liu, Wang, and Xie \cite{B} proved in 2025 the existence of $(-1)$-homogeneous solutions regular away from the origin for arbitrary locally Lipschitz $(-3)$-homogeneous forces when $4\le n\le16$. Their argument combines a relation between the radial velocity and the Bernoulli function with a power test on the sphere. They also establish uniqueness among all classical homogeneous solutions for sufficiently small force and prove existence in every dimension $n\ge4$ for nonnegative radial forces.

For the non-homogeneous whole-space problem, Wang and Yang \cite{WY} extended Li and Yang's existence theorem to every $n\ge5$ in 2026 and also treated periodic problems and a larger class of whole-space forces. Their improvement of the Bernoulli estimate motivates our argument. We combine the energy and Bernoulli estimates of Bang et al.\ with a finite-exponent estimate for an auxiliary vector field to remove the upper dimension restriction in the homogeneous existence theorem.

\begin{theorem}\label{thm:main}
Let $n\ge4$ be an integer and let $f:\R^n\setminus\{0\}\to\R^n$ be locally Lipschitz and homogeneous of degree $-3$.
\begin{enumerate}
\item There exists a classical solution $(u,p)$ of the stationary Navier--Stokes system above such that
\[
 u(\lambda x)=\lambda^{-1}u(x),\qquad p(\lambda x)=\lambda^{-2}p(x),\qquad \lambda>0.
\]
For every $0<\alpha<1$,
\begin{equation*}
 \|u\|_{C^{2,\alpha}(\Sph)}+\|p\|_{C^{1,\alpha}(\Sph)}
 \le C\bigl(n,\alpha,\|f\|_{\Lip(\Sph)}\bigr).
\end{equation*}
In particular, the solution is classical on $\mathbb{R}^n\setminus\{0\}$. If $f$ is smooth there, then $u$ and $p$ are smooth there.
\item There exists $\varepsilon_n>0$ such that, if $\|f\|_{\Lip(\Sph)}\le\varepsilon_n$, the velocity is unique among all degree $-1$ homogeneous $C^2$ solutions on $\mathbb{R}^n\setminus\{0\}$. The degree $-2$ pressure is uniquely determined.
\end{enumerate}
\end{theorem}

\noindent\emph{Independent work.}
Shortly before the author was about to post this manuscript on arXiv, we became aware that Gui, Liu, Wang, and Xie had posted the preprint \cite{GLWX26}, which establishes the same extension of the homogeneous existence theorem to arbitrarily high dimensions, together with uniqueness for small forces. The author then promptly posted the present manuscript on arXiv. The corresponding results in Theorem~\ref{thm:main} were obtained independently of \cite{GLWX26}.

The a priori estimates below apply to every classical homogeneous solution when $n\ge7$. They yield a construction for $n\ge17$ which, together with \cite[Theorem 1.1]{B} for $4\le n\le16$, proves Theorem~\ref{thm:main}. All constants may depend on the fixed dimension.

The weighted pressure estimate used here goes back to Frehse and R\r{u}\v{z}i\v{c}ka \cite{FR95}; see also Li and Yang \cite[Lemma 2.10]{LY} and the homogeneous version in \cite[Lemma 2.6]{B}. Our main observation is that the local integral bounds supplied by this estimate suffice to control an auxiliary tangential vector field $F$ in the Lebesgue space required by the Bernoulli test. We use the subcritical weight $|x-\omega|^{4+1/n-n}$, split the gradient of the Green function into dyadic annuli, and interpolate the contribution from each annulus between $L^1$ and $L^\infty$.

On the unit sphere we write
\[
 u=u^t+u^r e_r,\qquad f=f^t+f^r e_r,\qquad
 H=p+\frac12|u|^2,\qquad e_r=\frac{x}{|x|}.
\]
Here $u^t$ and $f^t$ are tangent to the sphere, and $H$ is the Bernoulli function. The norm $\|f\|_{\Lip(\Sph)}$ is the sum of the supremum norm and the Lipschitz seminorm with respect to geodesic distance. We use the same letters for the fields on $\R^n\setminus\{0\}$ and their restrictions to $\Sph$. Write $H_+=\max\{H,0\}$ and $H_-=\max\{-H,0\}$, and set
\begin{equation*}
 \theta=\frac{(n-2)(n-1)}{2(n-3)}
       =\frac n2+\frac1{n-3}.
\end{equation*}
The energy argument of \cite{B} gives
\begin{equation*}
\begin{aligned}
 &\energyintegral \le C_n\left[\forcenorm^2+
              \bigl(\headnorm+\forcenorm\bigr)^{3/2}\right].
\end{aligned}
\end{equation*}
This energy estimate depends on $\headnorm$ because the energy identity contains the signed term $(n-4)\int_{\Sph}Hu^r\dd\sigma$.

Let $B^{\Sph}_\rho(\omega)$ denote the geodesic ball of radius $\rho$ centered at $\omega$ on $\Sph$. The subcritical localized pressure identity and homogeneity imply, for a fixed $\rho_0>0$,
\begin{equation*}
 \growthsup
 \le C_n\bigl(\headnorm+\|u\|_{L^2(\Sph)}^2+\forcenorm\bigr).
\end{equation*}
By the Cauchy--Schwarz inequality, the integral of $|u|$ over $B^{\Sph}_\rho(\omega)$ is bounded by $C_n\rho^{n-2-\frac{1}{2n}}$ times the square root of the right-hand side. With $s=n-4-\frac{1}{n}$, the kernel estimate in Lemma~\ref{lem:sphere-flux} bounds $F$ in $L^q(\Sph)$ for every fixed $1\le q<2n+1$, including $q=\frac{(n-2)(n-1)}{3n-8}$. Testing the resulting inequality for $H$ with $H_+^{\frac{n-4}{2}}$ then gives
\begin{equation*}
\begin{aligned}
 \headnorm
 &\le C_n\forcenorm\bigl(1+\|u\|_{L^2(\Sph)}\bigr)\\
 &\quad+C_n\forcenorm\left(\growthsup\right)^{1/2}.
\end{aligned}
\end{equation*}
As in \cite[Lemma 2.3]{B}, this power test cancels the tangential transport and radial reaction terms. The estimate for $F$ therefore gives the required $L^\theta$ bound directly, without BMO estimates or an iteration to obtain an $L^\infty$ bound. The construction of $F$ uses the inverse Laplacian on the sphere, normalized to have zero mean; no inverse Laplacian of $H$ or $H_+$ is needed.

Combining these three estimates bounds $\headnorm$ by a constant plus terms involving its powers $1/2$ and $3/4$. Young's inequality absorbs these terms. Substituting the resulting bound into the energy inequality yields the following estimates, proved in Proposition~\ref{prop:apriori}:
\begin{align*}
 \headnorm&\le C_n\bigl(\forcenorm+\forcenorm^4\bigr),
       \\
 \energyintegral&\le C_n\bigl(\forcenorm^2+\forcenorm^8\bigr).
\end{align*}
Since $\theta>n/2$, these bounds allow us to apply the local regularity criterion of Li and Yang \cite[Proposition 2.3]{LY}. We then obtain existence using a compact homogeneous Stokes map and Leray--Schauder degree theory. For sufficiently small force, the energy bound, compactness, and the quadratic structure of the Stokes map yield uniqueness among all classical homogeneous solutions.

The same ideas also yield a simplified proof of the existence theorem for bounded compactly supported forces, established by Li and Yang \cite{LY} for $5\le n\le15$ and by Wang and Yang \cite{WY} for $n\ge16$. We state this result separately and prove it in Section~\ref{sec:whole-space}.

\begin{theorem}[Compactly supported forces]\label{thm:compact-force}
Let $n\ge5$, let $R_0>0$, and let $f\in L^\infty(\R^n;\R^n)$ satisfy $\supp f\subset B_{R_0}(0)$. There exists a regular solution of
\[
 -\Delta u+(u\cdot\nabla)u+\nabla p=f,
 \qquad\diver u=0\quad\text{in }\R^n,
\]
with $u\in W^{2,b}_{\mathrm{loc}}(\R^n)$ and $p\in W^{1,b}_{\mathrm{loc}}(\R^n)$ for every $1<b<\infty$, such that
\[
 |u(x)|\le\frac{C}{(1+|x|)^{n-2}},\qquad
 |\nabla u(x)|+|p(x)|\le\frac{C}{(1+|x|)^{n-1}}.
\]
Here $C$ depends only on $n$, $R_0$, and $\|f\|_{L^\infty(\R^n)}$. The solution is smooth if $f$ is smooth.
\end{theorem}

The simplification lies in the Bernoulli estimate. For the whole-space argument, $\theta$ is now a variable Lebesgue exponent. A single power test of the scalar inequality $-\Delta H+u\cdot\nabla H\le\diver F$ on $\R^n$ gives
\[
 \|H_+\|_{L^\theta(\R^n)}
 \le C_{n,\theta}\|F\|_{L^q(\R^n)},\qquad
 \frac1q=\frac1\theta+\frac1n,\qquad
 \frac{2n}{n-2}\le\theta<\infty,
\]
under the integrability and decay assumptions of Lemma~\ref{lem:euclidean-scalar}. Here $F$ denotes the auxiliary vector field, and the constant may depend on $\theta$. The kernel estimate in Lemma~\ref{lem:euclidean-kernel} controls its $L^q$ norm. Taking $\theta=n$ gives
\[
 \|H_+\|_{L^n(\R^n)}
 \le C_n\|F\|_{L^{\frac n2}(\R^n)}
 \le C_f\bigl(1+\|H_+\|_{L^n(\R^n)}^{\frac12}\bigr).
\]
This closes the Bernoulli estimate using a subcritical weighted potential of $|u|^2$. The argument avoids BMO estimates, the scalar $L^\infty$ estimate  in \cite[Lemma 2.2]{WY}, and the subsequent estimate of $(-\Delta)^{-1}H_+$ in \cite[Proposition 3.4]{WY}. In the construction of $F$, the inverse Laplacian acts on $|f||u|$. The local regularity criterion, exterior decay estimates, and construction by Leray--Schauder degree theory follow Li and Yang \cite{LY}.

\noindent\emph{Use of artificial intelligence.}
The author used ChatGPT during the preparation of this paper. The author has checked all mathematical arguments and details by hand and takes full responsibility for the content and accuracy of the paper.

The paper is organized as follows. Section~\ref{sec:apriori} proves the a priori estimates for homogeneous solutions. Section~\ref{sec:existence} proves Theorem~\ref{thm:main}, including regularity, existence by Leray--Schauder degree theory, and small-force uniqueness. Section~\ref{sec:whole-space} develops the analogous whole-space estimates and gives the simplified proof of Theorem~\ref{thm:compact-force} for compactly supported forces.

\section{A priori estimates for homogeneous solutions}\label{sec:apriori}

Throughout this section, we assume $n\ge7$ and let $(u,p)$ be an individually classical homogeneous solution of
\begin{equation}\label{eq:apriori-equation}
 -\Delta u+(u\cdot\nabla)u+\nabla p=f,
 \qquad\diver u=0.
\end{equation}
Unless stated otherwise, the constants depend only on $n$ and are independent of the particular solution. The force $f$ is locally Lipschitz and homogeneous of degree $-3$.

\begin{lemma}[Normalization of the pressure]\label{lem:pressure-normalization}
If $u$ and $f$ have degrees $-1$ and $-3$, respectively, the pressure has a unique representative homogeneous of degree $-2$.
\end{lemma}
\begin{proof}The momentum equation \eqref{eq:apriori-equation} implies
\begin{align*}
 \nabla p(rx)
 &=f(rx)+\Delta u(rx)-\bigl((u\cdot\nabla)u\bigr)(rx) \\
 & = r^{-3} f(x) + r^{-2} \Delta (u(rx)) - r^{-1} (u(rx) \cdot \nabla)(u(rx) ) \\
 & = r^{-3}(f(x) + \Delta u(x) - ((u\cdot \nabla) u) (x) ) = r^{-3} \nabla p(x)
\end{align*}
Thus the pressure gradient is homogeneous of degree $-3$.

To recover the pressure itself, fix $r>0$ and apply the chain rule in $x$:
\[
 \nabla_x\bigl(r^2p(rx)-p(x)\bigr)
       =r^3\nabla p(rx)-\nabla p(x)=0.
\]
The function in parentheses is therefore constant in $x$, because $\R^n\setminus\{0\}$ is connected. Fix a point $x_0\ne0$. For every $x\ne0$ and $r>0$,
\[
 r^2p(rx)-p(x)=r^2p(rx_0)-p(x_0).
\]
Differentiate this identity with respect to $r$ at $r=1$. The result is
\[
 2p(x)+x\cdot\nabla p(x)
       =2p(x_0)+x_0\cdot\nabla p(x_0)=:c.
\]
Here $c$ is a single constant on $\mathbb{R}^n \setminus \{0\}$. We can normalize $p$ by adding $-c/2$ to $p$.

Define
\[
 \widetilde p(x)=p(x)-\frac c2.
\]
Subtracting this physical constant leaves $\nabla p$ and the momentum equation unchanged. Moreover,
\[
 x\cdot\nabla\widetilde p(x)+2\widetilde p(x)
       =x\cdot\nabla p(x)+2p(x)-c=0.
\]
We now integrate this identity along a ray. For $\omega\in\Sph$ and $r>0$,
\begin{align*}
 \frac{\dd}{\dd r}\bigl(r^2\widetilde p(r\omega)\bigr)
 &=2r\widetilde p(r\omega)
       +r^2\omega\cdot\nabla\widetilde p(r\omega)\\
 &=r\bigl(2\widetilde p(r\omega)
       +(r\omega)\cdot\nabla\widetilde p(r\omega)\bigr)=0.
\end{align*}
Hence $r^2\widetilde p(r\omega)=\widetilde p(\omega)$, or
\[
 \widetilde p(r\omega)=r^{-2}\widetilde p(\omega).
\]
This is precisely degree $-2$ homogeneity. The normalization can also be written explicitly as
\[
 \widetilde p(x)=p(x)-\frac c2
       =-\frac12x\cdot\nabla p(x).
\]

Finally, any two pressures for the same velocity and force have the same gradient by the momentum equation. Their difference is therefore a constant, say $c_0$, on $\mathbb{R}^n\setminus\{0\}$. If both pressures are homogeneous of degree $-2$, then their difference obeys $c_0=r^{-2}c_0$ for every $r>0$. Taking $r=2$ gives $c_0=0$, proving uniqueness of the homogeneous representative.
\end{proof}

We use this representative from now on.  With the notation of Section~\ref{sec:introduction},
\[
 u(r\omega)=r^{-1}u(\omega),\qquad
 p(r\omega)=r^{-2}p(\omega),\qquad
 H(r\omega)=r^{-2}H(\omega).
\]

\begin{lemma}[Equations and Bernoulli identity]\label{lem:angular}
On $\Sph$ the solution satisfies
\begin{equation}\label{eq:divergence}
 \diver_{\Sph}u^t=-(n-2)u^r,
\end{equation}
\begin{equation}\label{eq:angular-momentum}
\begin{aligned}
 -\Delta_{\Sph}u+(n-3)u
 &+(u^t\cdot\nabla_{\Sph})u-u^r u+ \nabla_{\Sph}p-2pe_r=f,
\end{aligned}
\end{equation}
and in particular, the radial component of \eqref{eq:angular-momentum} gives
\begin{equation}\label{eq:radial}
 -\Delta_{\Sph}u^r+u^t\cdot\nabla_{\Sph}u^r=2H+f^r.
\end{equation}
Its Bernoulli function satisfies, distributionally,
\begin{equation}\label{eq:bernoulli}
\begin{aligned}
 -\Delta_{\Sph}H+u^t\cdot\nabla_{\Sph}H
 &+\bigl(2(n-4)-2u^r\bigr)H\\
 &=-\frac12\sum_{i,j=1}^n(\partial_i u_j-\partial_j u_i)^2 +\bigl(f\cdot u-\diver_{\Sph}f^t-(n-4)f^r\bigr).
\end{aligned}
\end{equation}
The Euclidean derivatives in the last equation are evaluated at radius one.
\end{lemma}
\begin{proof}
For a scalar function $\phi$ on $\Sph$ and a real exponent $\kappa$, polar differentiation gives
\begin{align*}
 \Delta\bigl(r^\kappa\phi(\omega)\bigr)
 &=r^{\kappa-2}\bigl(\Delta_{\Sph}\phi+\kappa(\kappa+n-2)\phi\bigr),\\
 \nabla\bigl(r^\kappa\phi(\omega)\bigr)
 &=r^{\kappa-1}\bigl(\nabla_{\Sph}\phi+\kappa\phi e_r\bigr).
\end{align*}
The Laplacian formula holds componentwise. The polar divergence formula proves \eqref{eq:divergence}. Also
\[
 (u\cdot\nabla)u
 =r^{-3}\bigl((u^t\cdot\nabla_{\Sph})u-u^r u\bigr),
\]
which gives \eqref{eq:angular-momentum}. Since $\Delta_{\Sph}e_r=-(n-1)e_r$ and $\nabla_{\Sph,e}e_r=e$ for tangent unit vectors $e$, the product rule yields
\begin{align*}
 (\Delta_{\Sph}u)\cdot e_r
 &=\Delta_{\Sph}u^r-2\diver_{\Sph}u^t-(n-1)u^r,\\
 \bigl((u^t\cdot\nabla_{\Sph})u\bigr)\cdot e_r
 &=u^t\cdot\nabla_{\Sph}u^r-|u^t|^2.
\end{align*}
Taking the radial component in \eqref{eq:angular-momentum} and substituting \eqref{eq:divergence} cancels the linear terms in $u^r$ and gives \eqref{eq:radial}. These identities coincide with \cite[(13), (15)]{B}.

Dotting equation \eqref{eq:apriori-equation} with $u$ and taking its divergence give
\begin{gather*}
 -\Delta\left(\frac{|u|^2}{2}\right)+|\nabla u|^2
       +u\cdot\nabla H=f\cdot u,\\
 -\Delta p=\sum_{i,j}\partial_i u_j\partial_j u_i-\diver f.
\end{gather*}
Since
\[
 |\nabla u|^2-\sum_{i,j}\partial_i u_j\partial_j u_i
 =\frac12\sum_{i,j}(\partial_i u_j-\partial_j u_i)^2,
\]
their sum is
\[
 -\Delta H+u\cdot\nabla H
 =-\frac12\sum_{i,j}(\partial_i u_j-\partial_j u_i)^2
       +f\cdot u-\diver f.
\]
Finally, $H$ has degree $-2$ and
\[
 \diver f=r^{-4}\bigl(\diver_{\Sph}f^t+(n-4)f^r\bigr).
\]
The polar Laplacian and transport formulas now give \eqref{eq:bernoulli}.
\end{proof}

The energy and radial-mean identities below are those of \cite[Lemmas 2.2 and 2.4]{B}. We include their derivation and the energy bound used in the closure argument.

Since $\partial_r u=-u$ at radius one, the radial and tangential derivatives are orthogonal in the Euclidean gradient norm. Thus
\begin{equation}\label{eq:energy-definition}
\begin{aligned}
 &\energyintegral =\int_{\Sph}\bigl(|\nabla u|^2+(n-4)|u|^2\bigr)\dd\sigma.
\end{aligned}
\end{equation}
This integral controls the full squared $H^1(\Sph)$ norm of $u$ because $n-3>0$.

Recall that
\[
 \theta=\frac{(n-2)(n-1)}{2(n-3)}
       =\frac{n}{2}+\frac{1}{n-3}>\frac{n}{2}.
\]

\begin{lemma}[Energy and pressure estimates]\label{lem:energy}
Every classical solution of \eqref{eq:apriori-equation} satisfies
\begin{equation}\label{eq:pressure-one}
 \|p\|_{L^1(\Sph)}
 \le C_n\bigl(\headnorm+\|u\|_{L^2(\Sph)}^2+\forcenorm\bigr).
\end{equation}
It also satisfies the energy estimate
\begin{equation}\label{eq:energy-pre}
\begin{aligned}
 &\energyintegral\\
 &\qquad\le C_n\left[
   \forcenorm^2+\bigl(\headnorm+\forcenorm\bigr)^{3/2}\right].
\end{aligned}
\end{equation}
\end{lemma}

\begin{proof}
Integrating \eqref{eq:radial} eliminates its Laplacian. Integration by parts and \eqref{eq:divergence} give
\[
 \int_{\Sph}u^t\cdot\nabla_{\Sph}u^r\dd\sigma
 =-\int_{\Sph}(\diver_{\Sph}u^t)u^r\dd\sigma
 =(n-2)\int_{\Sph}(u^r)^2\dd\sigma.
\]
Writing $H=H_+-H_-$ gives the radial-mean identity
\begin{equation}\label{eq:radial-mean}
\begin{aligned}
 (n-2)\int_{\Sph}(u^r)^2\dd\sigma
       +2\int_{\Sph}H_-\dd\sigma
 &=2\int_{\Sph}H_+\dd\sigma +\int_{\Sph}f^r\dd\sigma.
\end{aligned}
\end{equation}
Using $p=H-|u|^2/2$ yields
\[
 \|p\|_{L^1(\Sph)}
 \le 2\int_{\Sph}H_+\dd\sigma
      +\frac12\int_{\Sph}|f^r|\dd\sigma
      +\frac12\|u\|_{L^2(\Sph)}^2.
\]
H\"older's inequality on the finite sphere proves \eqref{eq:pressure-one}. 

Next take the inner product of \eqref{eq:angular-momentum} with $u$ and integrate. The Laplacian and linear velocity term give the integral in \eqref{eq:energy-definition}. The transport terms give
\begin{align*}
 \int_{\Sph}\bigl((u^t\cdot\nabla_{\Sph})u-u^r u\bigr)
                         \cdot u\dd\sigma
 & =-\frac12\int_{\Sph}(\diver_{\Sph}u^t)|u|^2\dd\sigma
                    -\int_{\Sph}u^r|u|^2\dd\sigma\\
 & =\frac{n-4}{2}\int_{\Sph}u^r|u|^2\dd\sigma.
\end{align*}
For the pressure terms, tangential integration by parts gives
\begin{align*}
 \int_{\Sph}(\nabla_{\Sph}p-2p e_r)\cdot u\dd\sigma
 &=-\int_{\Sph}p\diver_{\Sph}u^t\dd\sigma
                    -2\int_{\Sph}pu^r\dd\sigma\\
 &=(n-4)\int_{\Sph}pu^r\dd\sigma.
\end{align*}
Adding these two expressions gives the exact energy identity
\begin{equation}\label{eq:exact-energy}
\begin{aligned}
 &\energyintegral+(n-4)\int_{\Sph}Hu^r\dd\sigma =\int_{\Sph}f\cdot u\dd\sigma.
\end{aligned}
\end{equation}

To estimate the potentially signed term $Hu^r$, test \eqref{eq:radial} with $(u^r)_+=\max(u^r,0)$. The diffusion term is the squared gradient of this positive part. The transport term equals
\[
 -\frac12\int_{\Sph}(\diver_{\Sph}u^t)(u^r)_+^2\dd\sigma
 =\frac{n-2}{2}\int_{\Sph}(u^r)_+^3\dd\sigma.
\]
Hence
\begin{align*}
 &\int_{\Sph}|\nabla_{\Sph}(u^r)_+|^2\dd\sigma
       +\frac{n-2}{2}\int_{\Sph}(u^r)_+^3\dd\sigma =2\int_{\Sph}H(u^r)_+\dd\sigma
           +\int_{\Sph}f^r(u^r)_+\dd\sigma.
\end{align*}
The Lipschitz positive-part test is justified by smooth approximation in the weak equation. In particular,
\[
 \int_{\Sph}H(u^r)_+\dd\sigma
 \ge-\frac12\int_{\Sph}f^r(u^r)_+\dd\sigma,
 \qquad
 \int_{\Sph}H(u^r)_-\dd\sigma
 \le\int_{\Sph}H_+(u^r)_-\dd\sigma.
\]
Splitting $u^r=(u^r)_+-(u^r)_-$ in the exact energy identity proves
\begin{equation}\label{eq:energy-inequality}
\begin{aligned}
 \energyintegral \le & \int_{\Sph}f\cdot u\dd\sigma
       +(n-4)\int_{\Sph}H_+(u^r)_-\dd\sigma\\
 &\qquad\quad+\frac{n-4}{2}
                         \int_{\Sph}f^r(u^r)_+\dd\sigma.
\end{aligned}
\end{equation}

Because $\theta>2$, its H\"older conjugate is less than $2$. Finite surface measure therefore implies
\begin{align*}
 \int_{\Sph}H_+(u^r)_-\dd\sigma
 &\le C_n\headnorm\|u^r\|_{L^2(\Sph)},\\
 \left|\int_{\Sph}f^r(u^r)_+\dd\sigma\right|
 &\le C_n\forcenorm\|u^r\|_{L^2(\Sph)}.
\end{align*}
The remaining force term satisfies
\[
 \left|\int_{\Sph}f\cdot u\dd\sigma\right|
 \le C_n\forcenorm\left(\energyintegral\right)^{1/2}.
\]
The radial-mean identity \eqref{eq:radial-mean}, after discarding its nonnegative $H_-$ term, gives by H\"older's inequality
\[
 \|u^r\|_{L^2(\Sph)}
 \le C_n\bigl(\headnorm+\forcenorm\bigr)^{1/2}.
\]
Using this radial-velocity bound in \eqref{eq:energy-inequality} gives
\begin{align*}
 &\energyintegral\\
 &\quad\le C_n\forcenorm\left(\energyintegral\right)^{1/2}\\
 &\qquad\quad+C_n\bigl(\headnorm+\forcenorm\bigr)^{3/2}\\
 &\quad\le\frac12\energyintegral+C_n\forcenorm^2\\
 &\qquad\quad+C_n\bigl(\headnorm+\forcenorm\bigr)^{3/2}.
\end{align*}
Absorbing half of the integral proves \eqref{eq:energy-pre}.

\end{proof}

The weighted estimate below is the homogeneous counterpart of Li and Yang \cite[Lemma 2.10]{LY}. In the paragraph preceding that lemma, they attribute this estimate and their Lemma 2.11 to Frehse and R\r{u}\v{z}i\v{c}ka \cite[Corollary and Theorem 4.1, p.~137]{FR95}. The weighted velocity estimate in the homogeneous setting is already contained in the proof of Bang, Gui, Liu, Wang, and Xie \cite[Lemma 2.6]{B}. We include the derivation with the lower-order bounds suited to the present setting. 

\begin{lemma}[A subcritical weighted pressure estimate]\label{lem:local-pressure}
Let $n\ge7$, let $(u,p)$ be a classical homogeneous solution under consideration, and fix $R=1/32$. For every $0<s<n-4$, uniformly in $\omega\in\Sph$,
\begin{equation}\label{eq:local-pressure-bound}
 \int_{B_R(\omega)}
       \frac{|u(x)|^2}{|x-\omega|^{s+2}}\dd x
 \le C_{n,s}\bigl(\headnorm+\|u\|_{L^2(\Sph)}^2+\forcenorm\bigr).
\end{equation}
The constant depends only on $n$ and $s$ and is independent of the particular solution.
\end{lemma}

\begin{proof}
Fix $0<s<n-4$ and $\omega\in\Sph$. Choose $\zeta\in C_c^\infty(B_{2R}(\omega))$ such that $0\le\zeta\le1$, $\zeta=1$ on $B_R(\omega)$, and
\[
 |\nabla\zeta|\le C_{n,s}R^{-1},\qquad
 |\nabla^2\zeta|\le C_{n,s}R^{-2}.
\]
Its support lies in the fixed annulus $\{1/2<|x|<3/2\}$. Put
\[
 k(x)=|x-\omega|^{-s}.
\]
Direct differentiation gives
\begin{align*}
 \partial_i k
 &=-s
       \frac{x_i-\omega_i}{|x-\omega|^{s+2}},\\
 \partial_{ij}k
 &=s(s+2)
       \frac{(x_i-\omega_i)(x_j-\omega_j)}{|x-\omega|^{s+4}}
       -s
       \frac{\delta_{ij}}{|x-\omega|^{s+2}},\\
 \Delta k
 &=-s(n-2-s)
                                      |x-\omega|^{-s-2}.
\end{align*}
The distributional pressure equation is
\[
 -\Delta p=\partial_i\partial_j(u_i u_j)-\diver f.
\]
Consequently, every smooth test function $\psi$ compactly supported away from the origin satisfies
\begin{equation}\label{eq:subcritical-pressure-test}
 \int_{\R^n}u_i u_j\partial_{ij}\psi\dd x
 +\int_{\R^n}p\Delta\psi\dd x
 =-\int_{\R^n}f\cdot\nabla\psi\dd x.
\end{equation}
Repeated Cartesian indices are summed from $1$ to $n$.

To use the singular test $\psi=\zeta k$, first replace $k$ by
\[
 k_\tau(x)=\bigl(|x-\omega|^2+\tau^2\bigr)^{-s/2},
 \qquad \tau>0.
\]
The functions $\zeta k_\tau$ are smooth and compactly supported away from the origin. The kernels and their first two derivatives converge away from $\omega$ and are bounded, respectively, by constant multiples of
\[
 |x-\omega|^{-s},\qquad
 |x-\omega|^{-s-1},\qquad
 |x-\omega|^{-s-2}.
\]
Their radial integrands near $\omega$ are $r^{n-1-s}$, $r^{n-2-s}$, and $r^{n-3-s}$, which are integrable because $s<n-4$. Dominated convergence, including the terms obtained by differentiating $\zeta$, gives $\zeta k_\tau\to\zeta k$ in $W^{2,1}$. The individual local boundedness of $u,p,f$ therefore passes \eqref{eq:subcritical-pressure-test} to the limit. No uniform pointwise bound for $u$ or $p$ is assumed, and no second classical derivative of $p$ is required.

Expand the derivatives of $\zeta k$ in \eqref{eq:subcritical-pressure-test} and substitute $p=H-|u|^2/2$. The resulting identity is
\begin{equation}\label{eq:subcritical-pressure-identity}
\begin{aligned}
 &s(s+2)\int_{\R^n}
       \frac{|u\cdot(x-\omega)|^2\zeta}{|x-\omega|^{s+4}}\dd x
       +\frac{s(n-4-s)}2
       \int_{\R^n}\frac{|u|^2\zeta}{|x-\omega|^{s+2}}\dd x\\
 &\qquad=s(n-2-s)
       \int_{\R^n}\frac{H\zeta}{|x-\omega|^{s+2}}\dd x
       +\mathcal R_\omega,
\end{aligned}
\end{equation}
where the remainder is exactly
\begin{align*}
 \mathcal R_\omega={}&-\int_{\R^n}pk\Delta\zeta\dd x
 -2\int_{\R^n}p\nabla k\cdot\nabla\zeta\dd x
 -\int_{\R^n}f\cdot\nabla(\zeta k)\dd x\\
 &-2\int_{\R^n}u_i u_j(\partial_i k)(\partial_j\zeta)\dd x
 -\int_{\R^n}u_i u_j k\partial_{ij}\zeta\dd x.
\end{align*}
The coefficient of the isotropic velocity term is positive because
\[
 s
       \left[-1+\frac12(n-2-s)\right]
       =\frac{s(n-4-s)}2>0.
\]

All differentiated-cutoff terms are supported where $R\le|x-\omega|\le2R$, and their kernels are bounded in terms of $n$, $s$, and the fixed radius. The force term that differentiates $k$ is also bounded, since
\[
 \int_{B_{2R}(\omega)}|\nabla k|\dd x
 \le C_{n,s}\int_0^{2R}r^{n-2-s}\dd r<\infty.
\]
It follows that
\begin{align*}
 |\mathcal R_\omega|
 &\le C_{n,s}\bigl(
    \|p\|_{L^1(B_{2R}(\omega))}
       +\|u\|_{L^2(B_{2R}(\omega))}^2
       +\|f\|_{L^\infty(B_{2R}(\omega))}\bigr).
\end{align*}
Homogeneity converts the annular norms to spherical norms. The radial integrals are finite dimensional constants, and the force supremum on this annulus is at most $8\forcenorm$. Applying \eqref{eq:pressure-one} therefore yields
\begin{equation}\label{eq:subcritical-remainder}
 |\mathcal R_\omega|
 \le C_{n,s}\bigl(\headnorm+\|u\|_{L^2(\Sph)}^2+\forcenorm\bigr).
\end{equation}

For the head term, H\"older's inequality gives
\begin{align*}
 \int_{B_{2R}(\omega)}
       \frac{H_+(x)}{|x-\omega|^{s+2}}\dd x
 &\le\|H_+\|_{L^\theta(B_{2R}(\omega))}
    \left(\int_{B_{2R}(\omega)}
       |x-\omega|^{-(s+2)\theta/(\theta-1)}\dd x
    \right)^{(\theta-1)/\theta}\\
 &\quad\le C_{n,s}\headnorm.
\end{align*}
The kernel integral is finite precisely when $\theta>n/(n-2-s)$; this follows from $n-2-s>2$ and $\theta>n/2$ by the definition of $\theta$. Discard the first nonnegative integral in \eqref{eq:subcritical-pressure-identity}, replace $H$ by $H_+$, and use \eqref{eq:subcritical-remainder}. Since $\zeta=1$ on $B_R(\omega)$, division by the positive coefficient $s(n-4-s)/2$ proves \eqref{eq:local-pressure-bound}.
\end{proof}

\begin{lemma}\label{lem:pressure-growth}
Fix $\rho_0=1/1024$, and let $B^{\Sph}_\rho(\omega)$ denote the geodesic ball of radius $\rho$ in $\Sph$. For every $0<s<n-4$,
\begin{equation}\label{eq:cap-growth}
 \sup_{\substack{\omega\in\Sph\\0<\rho\le\rho_0}}
       \rho^{-s-1}\int_{B^{\Sph}_\rho(\omega)}|u|^2\dd\sigma
 \le C_{n,s}\bigl(\headnorm+\|u\|_{L^2(\Sph)}^2+\forcenorm\bigr).
\end{equation}
\end{lemma}

\begin{proof}
Fix $0<s<n-4$. Because $s+2>0$, for $0<r\le R=1/32$ one has
\[
 r^{-s-2}\int_{B_r(\omega)}|u|^2\dd x
 \le\int_{B_r(\omega)}
       \frac{|u(x)|^2}{|x-\omega|^{s+2}}\dd x.
\]
Lemma~\ref{lem:local-pressure} therefore gives, uniformly in $\omega\in\Sph$ and $0<r\le1/32$,
\begin{equation}\label{eq:ambient-growth}
 r^{-s-2}\int_{B_r(\omega)}|u|^2\dd x
 \le C_{n,s}\bigl(\headnorm+\|u\|_{L^2(\Sph)}^2+\forcenorm\bigr).
\end{equation}

 Homogeneity and polar integration give
\begin{align*}
 \int_{B_{4\rho}(\omega)}|u|^2\dd x
 &\ge c_n\rho\int_{B^{\Sph}_\rho(\omega)}|u|^2\dd\sigma.
\end{align*}
The last constant is uniform because $\rho\le1/1024$. Consequently,
\begin{align*}
 \rho^{-s-1}\int_{B^{\Sph}_\rho(\omega)}|u|^2\dd\sigma
 &\le C_{n,s}\rho^{-s-2}
                     \int_{B_{4\rho}(\omega)}|u|^2\dd x\\
 &\quad\le C_{n,s}\bigl(\headnorm+\|u\|_{L^2(\Sph)}^2+\forcenorm\bigr).
\end{align*}
The last line uses \eqref{eq:ambient-growth} at radius $4\rho\le1/256<1/32$; the factor $4^{s+2}$ is absorbed into $C_{n,s}$. Taking the supremum over the centers and radii proves \eqref{eq:cap-growth}.
\end{proof}

For $n\ge7$, set
\[
 q=\frac{(n-2)(n-1)}{3n-8}.
\]
The identities
\begin{equation}\label{eq:flux-exponent}
 q-2=\frac{(n-3)(n-6)}{3n-8}>0,
 \qquad \frac2q+\frac{n-6}{2\theta}=1
\end{equation}
give the H\"older exponents in the proposition below. The range of $s$ in the following proposition ensures that this $q$ is admissible in Lemma~\ref{lem:sphere-flux}.

\begin{proposition}[The positive Bernoulli estimate]\label{prop:head}
Let $n\ge7$, and fix
\[
 n-4-\frac2{q-1}<s<n-4.
\]
Every classical homogeneous Navier--Stokes solution under consideration satisfies
\begin{equation}\label{eq:head-unclosed}
\begin{aligned}
 \headnorm
 &\le C_{n,s}\forcenorm\bigl(1+\|u\|_{L^2(\Sph)}\bigr)\\
 &\quad+C_{n,s}\forcenorm\left(\growthsup[-s-1]\right)^{1/2}.
\end{aligned}
\end{equation}
The constants depend only on $n$ and $s$.
\end{proposition}

The proof uses the following estimate for an auxiliary vector field. Its proof is given immediately after the proof of Proposition~\ref{prop:head}.

\begin{lemma}\label{lem:sphere-flux}
Let $n>4$, let $0<s<n-4$ and $g\in L^1(\Sph)$, and fix $0<\rho_0<1/10$. Suppose
\[
 \sup_{\substack{z\in\Sph\\0<\rho\le\rho_0}}
 \rho^{-(n+s)/2}\int_{B^{\Sph}_\rho(z)}|g|\dd\sigma<\infty.
\]
With $\bar g=|\Sph|^{-1}\int_{\Sph}g\dd\sigma$, every fixed
\[
 1\le q<\frac{n-2-s}{n-4-s}=1+\frac2{n-4-s}
\]
satisfies
\begin{equation}\label{eq:sphere-flux-finite}
\begin{aligned}
 &\left\|\nabla_{\Sph}(-\Delta_{\Sph})^{-1}(g-\bar g)
                       \right\|_{L^q(\Sph)}\\
 &\qquad\le C_{n,s,q,\rho_0}\left[
 \|g\|_{L^1(\Sph)}+
 \sup_{\substack{z\in\Sph\\0<\rho\le\rho_0}}
 \rho^{-(n+s)/2}\int_{B^{\Sph}_\rho(z)}|g|\dd\sigma\right].
\end{aligned}
\end{equation}
The inverse is normalized to have zero mean, and the constant depends only on $n,s,q,$ and $\rho_0$.
\end{lemma}

\begin{proof}[Proof of Proposition~\ref{prop:head}]
Define the nonnegative source $g$, its mean $b$, and the auxiliary tangential vector field $F$ by
\begin{align*}
 g&=|f||u|+(n-4)|f^r|,
 \qquad b=\frac1{|\Sph|}\int_{\Sph}g\dd\sigma,\\
 F&=-\nabla_{\Sph}(-\Delta_{\Sph})^{-1}(g-b)-f^t.
\end{align*}
The inverse is normalized to have zero mean on the sphere. By the definition of $F$,
\[
 b+\diver_{\Sph}F=g-\diver_{\Sph}f^t.
\]
Dropping the nonpositive vorticity term in \eqref{eq:bernoulli} and using $f\cdot u\le |f||u|$ and $-(n-4)f^r\le(n-4)|f^r|$, we obtain
\begin{equation}\label{eq:scalar-equation}
 -\Delta_{\Sph}H+u^t\cdot\nabla_{\Sph}H
     +\bigl(2(n-4)-2u^r\bigr)H
 \le b+\diver_{\Sph}F.
\end{equation}
\emph{Step 1.} We prove that
\begin{equation}\label{eq:scalar-bound}
 \headnorm
 \le C_n\left(b+\|F\|_{L^q(\Sph)}\right).
\end{equation}
For each fixed classical solution, $g$ and $f^t$ are bounded on the sphere. The assumed range of $s$ ensures $0<s<n-4$ and $q<(n-2-s)/(n-4-s)$. Thus the growth quantity for $g$ in Lemma~\ref{lem:sphere-flux} is finite, and the lemma gives $F\in L^q(\Sph)$.

We test \eqref{eq:scalar-equation} with $H_+^{(n-4)/2}$. This test is admissible by nonnegative smooth approximation: both $H_+^{(n-4)/2}$ and $H_+^{(n-2)/4}$ are $C^1$, including at $H=0$, since their exponents exceed one for $n\ge7$. Integration by parts and \eqref{eq:divergence} give
\begin{align*}
 \int_{\Sph}(u^t\cdot\nabla_{\Sph}H)H_+^{(n-4)/2}\dd\sigma
 &=-\frac2{n-2}\int_{\Sph}(\diver_{\Sph}u^t)
                                      H_+^{(n-2)/2}\dd\sigma\\
 &=2\int_{\Sph}u^rH_+^{(n-2)/2}\dd\sigma.
\end{align*}
This term cancels the contribution of $-2u^rH$.

The chain rule therefore gives
\begin{align*}
 &\frac{8(n-4)}{(n-2)^2}
       \int_{\Sph}\left|\nabla_{\Sph}H_+^{(n-2)/4}\right|^2\dd\sigma
       +2(n-4)\int_{\Sph}H_+^{(n-2)/2}\dd\sigma\\
 &\quad\le b\int_{\Sph}H_+^{(n-4)/2}\dd\sigma
       -\frac{2(n-4)}{n-2}
       \int_{\Sph}F\cdot\nabla_{\Sph}H_+^{(n-2)/4}
                                      H_+^{(n-6)/4}\dd\sigma.
\end{align*}
Young's inequality gives the pointwise bound
\begin{align*}
 &\frac{2(n-4)}{n-2}|F|
     \left|\nabla_{\Sph}H_+^{(n-2)/4}\right|H_+^{(n-6)/4}\\
 &\quad\le\frac{4(n-4)}{(n-2)^2}
        \left|\nabla_{\Sph}H_+^{(n-2)/4}\right|^2
            +\frac{n-4}{4}|F|^2H_+^{(n-6)/2}.
\end{align*}
After absorbing the gradient term, the left-hand side controls the squared $H^1(\Sph)$ norm of $H_+^{(n-2)/4}$. Sobolev embedding on $\Sph$ then yields
\begin{align*}
 \headnorm^{(n-2)/2}
 &=\left\|H_+^{(n-2)/4}\right\|_{L^{2(n-1)/(n-3)}(\Sph)}^2\\
 &\le C_n\left[
        b\int_{\Sph}H_+^{(n-4)/2}\dd\sigma
          +\int_{\Sph}|F|^2H_+^{(n-6)/2}\dd\sigma\right].
\end{align*}
The equality follows from the definition of $\theta$. H\"older's inequality bounds the first integral by $C_n\headnorm^{(n-4)/2}$. For the second, the identity $2/q+(n-6)/(2\theta)=1$ in \eqref{eq:flux-exponent} gives
\[
 \int_{\Sph}|F|^2H_+^{(n-6)/2}\dd\sigma
 \le\|F\|_{L^q(\Sph)}^2\headnorm^{(n-6)/2}.
\]
Consequently,
\begin{equation}\label{eq:scalar-powers}
 \headnorm^{(n-2)/2}
 \le C_n b\,\headnorm^{(n-4)/2}
      +C_n\|F\|_{L^q(\Sph)}^2\headnorm^{(n-6)/2}.
\end{equation}
If $\headnorm=0$, the desired bound is immediate. Otherwise, dividing \eqref{eq:scalar-powers} by $\headnorm^{(n-6)/2}$ gives
\[
 \headnorm^2
 \le C_n b\,\headnorm
          +C_n\|F\|_{L^q(\Sph)}^2.
\]
Applying Young's inequality to $C_n b\,\headnorm$, absorbing half of $\headnorm^2$, and taking square roots proves \eqref{eq:scalar-bound}.

\emph{Step 2.} We prove that
\begin{equation}\label{eq:flux-bound}
\begin{aligned}
 &b+\|F\|_{L^q(\Sph)}
 \le C_{n,s}\forcenorm\bigl(1+\|u\|_{L^2(\Sph)}\bigr)\\
 &\qquad+C_{n,s}\forcenorm\left(\growthsup[-s-1]\right)^{1/2}.
\end{aligned}
\end{equation}

To estimate $F$, apply Lemma~\ref{lem:sphere-flux} to $g$ and use the definition of $F$ together with $\|f^t\|_{L^q(\Sph)}\le C_n\forcenorm$. Since $q$ depends only on $n$ and $\rho_0$ is fixed, this gives
\begin{align*}
 \|F\|_{L^q(\Sph)}
 &\le C_{n,s}\biggl[\|g\|_{L^1(\Sph)}+\forcenorm\\
 &\qquad+\sup_{\substack{\omega\in\Sph\\0<\rho\le\rho_0}}
       \rho^{-(n+s)/2}\int_{B^{\Sph}_\rho(\omega)}g\dd\sigma\biggr].
\end{align*}
The definition of $g$ and H\"older's inequality give
\begin{equation}\label{eq:source-mass}
 \|g\|_{L^1(\Sph)}
 \le C_n\forcenorm\bigl(1+\|u\|_{L^2(\Sph)}\bigr).
\end{equation}
For every $\omega\in\Sph$ and $0<\rho\le\rho_0$, the bound $|B^{\Sph}_\rho(\omega)|\le C_n\rho^{n-1}$ and H\"older's inequality give
\begin{equation}\label{eq:source-growth}
\begin{aligned}
 &\rho^{-(n+s)/2}\int_{B^{\Sph}_\rho(\omega)}g\dd\sigma\\
 &\quad\le C_n\forcenorm
   \left[\rho^{-(n+s)/2}\int_{B^{\Sph}_\rho(\omega)}|u|\dd\sigma
         +\rho^{(n-2-s)/2}\right]\\
 &\quad\le C_n\forcenorm
   \left[\left(\rho^{-s-1}\int_{B^{\Sph}_\rho(\omega)}|u|^2\dd\sigma\right)^{1/2}
         +\rho^{(n-2-s)/2}\right]\\
 &\quad\le C_n\forcenorm
   \left[1+\left(\growthsup[-s-1]\right)^{1/2}\right].
\end{aligned}
\end{equation}
Here the last inequality uses $(n-2-s)/2>1$ and $\rho\le\rho_0<1$.
Taking the supremum over $\omega$ and $\rho$ in \eqref{eq:source-growth}, inserting the resulting bound and \eqref{eq:source-mass} into the preceding estimate for $F$, and using $b=|\Sph|^{-1}\|g\|_{L^1(\Sph)}$ proves \eqref{eq:flux-bound}.
Combining \eqref{eq:scalar-bound} with \eqref{eq:flux-bound} proves \eqref{eq:head-unclosed}.
\end{proof}

\begin{proof}[Proof of Lemma~\ref{lem:sphere-flux}]
The proof follows the strategy of Wang and Yang \cite[Lemma 2.4]{WY}, which applies a near--far decomposition. Here we adapt this strategy to the sphere and obtain finite $L^q$ estimates by interpolating the $L^1$ and $L^\infty$ bounds on each dyadic annulus.

We use the standard Green representation for the nonnegative operator $-\Delta_{\Sph}$; see \cite[Theorem~4.13, pp.~108--109]{Aubin}. For $n>3$, its mean-zero Green function is symmetric and smooth away from the diagonal, and satisfies
\begin{align*}
 -\Delta_{\Sph,x}G(x,y)&=\delta_y-|\Sph|^{-1},
       &G(x,y)&=G(y,x),\\
 \int_{\Sph}G(x,y)\dd\sigma_x&=0,
       &\int_{\Sph}G(x,y)\dd\sigma_y&=0.
\end{align*}
Writing $d(x,y)$ for geodesic distance and $K(x,y)=\nabla_{\Sph,x}G(x,y)$, the same theorem gives
\[
 |G(x,y)|\le C_n\bigl(1+d(x,y)^{3-n}\bigr),
 \qquad
 |K(x,y)|\le C_n d(x,y)^{2-n}.
\]
The gradient bound is integrable on the sphere. Differentiating the zero-mean normalization in the distributional sense therefore gives
\[
 \int_{\Sph}K(x,y)\dd\sigma_y=0.
\]

The Green integral
\[
 (-\Delta_{\Sph})^{-1}(g-\bar g)(x)
       =\int_{\Sph}G(x,y)g(y)\dd\sigma_y
\]
defines an $L^1$ function of zero mean. Denote its gradient by $Tg$, so that
\[
 Tg(x):=\int_{\Sph}K(x,y)g(y)\dd\sigma_y.
\]
The mean $\bar g$ has zero contribution by kernel cancellation. The bound on $K$ is also integrable in $x$, uniformly in $y$, since its radial integral near the diagonal is bounded by
\[
 \int_{\Sph}|K(x,y)|\dd\sigma_x
 \le C_n\int_0^1 t^{2-n}t^{n-2}\dd t=C_n.
\]
Tonelli therefore gives $\|Tg\|_{L^1(\Sph)}\le C_n\|g\|_{L^1(\Sph)}$. This proves the case $q=1$ and justifies the gradient representation. The case $g=0$ is immediate; suppose henceforth that $\|g\|_{L^1(\Sph)}>0$ and $1<q<(n-2-s)/(n-4-s)$.

We split $Tg$ into its near and far parts:
\[
 Tg(x)=\int_{B^{\Sph}_{\rho_0}(x)}K(x,y)g(y)\dd\sigma_y
       +\int_{\Sph\setminus B^{\Sph}_{\rho_0}(x)}K(x,y)g(y)\dd\sigma_y.
\]

For each integer $j\ge0$, define the nonnegative shell contribution
\[
 T_j(x)=\int_{2^{-j-1}\rho_0<d(x,y)\le2^{-j}\rho_0}
                  d(x,y)^{2-n}|g(y)|\dd\sigma_y.
\]
On this shell, $d(x,y)^{2-n}\le(2^{-j-1}\rho_0)^{2-n}$. The assumed growth bound therefore gives
\[
 \|T_j\|_{L^\infty(\Sph)}
 \le C_n(2^{-j}\rho_0)^{-(n-4-s)/2}
 \times\sup_{\substack{z\in\Sph\\0<\rho\le\rho_0}}
       \rho^{-(n+s)/2}\int_{B^{\Sph}_\rho(z)}|g|\dd\sigma.
\]
For the $L^1$ bound, Tonelli permits integration in $x$ first. Geodesic polar coordinates centered at $y$ give
\[
 \int_{2^{-j-1}\rho_0<d(x,y)\le2^{-j}\rho_0}
               d(x,y)^{2-n}\dd\sigma_x
 \le C_n\int_{2^{-j-1}\rho_0}^{2^{-j}\rho_0}\dd t
 \le C_n2^{-j}\rho_0,
\]
Consequently,
\[
 \|T_j\|_{L^1(\Sph)}
 \le C_n2^{-j}\rho_0\|g\|_{L^1(\Sph)}.
\]
The elementary interpolation inequality
\[
 \|v\|_{L^q}\le\|v\|_{L^\infty}^{1-1/q}
                         \|v\|_{L^1}^{1/q}
\]
now yields
\begin{align*}
 \|T_j\|_{L^q(\Sph)}
 &\le C_n(2^{-j}\rho_0)^{\frac{n-2-s}{(2q)}-\frac{n-4-s}{2}}
                   \|g\|_{L^1(\Sph)}^{1/q}\\
 &\quad\times\left[
 \sup_{\substack{z\in\Sph\\0<\rho\le\rho_0}}
 \rho^{-(n+s)/2}\int_{B^{\Sph}_\rho(z)}|g|\dd\sigma
 \right]^{1-1/q}.
\end{align*}
The exponent of $2^{-j}\rho_0$ is positive precisely when $q<(n-2-s)/(n-4-s)$. Thus the geometric series in $j$ converges, and summing over $j$ controls the part with $d(x,y)\le\rho_0$:
\begin{align*}
      \left\|\int_{d(x,y) \le \rho_0} K(x,y) g(y) \dd \sigma_y \right\|_{L^q} & \le C_n \sum_{j=0}^{+\infty} \| T_j\|_{L^q} \\
      & \le C_{n,q,s,\rho_0} \left(\|g\|_{L^1} + \sup_{\substack{z\in\Sph\\0<\rho\le\rho_0}}
 \rho^{-(n+s)/2}\int_{B^{\Sph}_\rho(z)}|g|\dd\sigma \right) 
\end{align*}
On $d(x,y)>\rho_0$, the kernel $K(x,y)$ is bounded in terms of $n$ and $\rho_0$. Since the sphere has finite volume, the corresponding part has $L^q$ norm at most $C_{n,s,q,\rho_0}\|g\|_{L^1(\Sph)}$. Combining the near and far estimates bounds $\|Tg\|_{L^q(\Sph)}$ and proves \eqref{eq:sphere-flux-finite}.
\end{proof}

\begin{proposition}\label{prop:apriori}
For each fixed $n\ge7$, the classical homogeneous solutions satisfy
\begin{equation}\label{eq:head-final}
 \headnorm\le C_n\bigl(\forcenorm+\forcenorm^4\bigr)
\end{equation}
and
\begin{equation}\label{eq:energy-final}
\begin{aligned}
 &\energyintegral \le C_n\bigl(\forcenorm^2+\forcenorm^8\bigr).
\end{aligned}
\end{equation}
The constants depend only on $n$ and are independent of the particular solution.
\end{proposition}

\begin{proof}
Fix $s=n-4-1/n$. For this choice,
\[
 q=\frac{(n-2)(n-1)}{3n-8}<n<2n+1
   =\frac{n-2-s}{n-4-s},
\]
so Proposition~\ref{prop:head} applies, and its constants $C_{n,s}$ depend only on $n$. With the same choice, \eqref{eq:cap-growth} of Lemma~\ref{lem:pressure-growth} gives
\begin{equation*}
 \growthsup \le C_n\bigl(\headnorm+\|u\|_{L^2(\Sph)}^2+\forcenorm\bigr).
\end{equation*}
Dropping the nonnegative gradient term from the energy integral gives
\[
 \|u\|_{L^2(\Sph)}^2\le\frac1{n-3}\energyintegral.
\]
Insert these two bounds into \eqref{eq:head-unclosed} and use the inequality that the square root of a sum is at most the sum of its square roots. We obtain
\begin{align*}
 \headnorm
 &\le C_n\forcenorm
          \left[1+\headnorm^{1/2}+\forcenorm^{1/2}\right]\\
 &\quad+C_n\forcenorm\left(\energyintegral\right)^{1/2}.
\end{align*}
The preliminary bound \eqref{eq:energy-pre} in Lemma~\ref{lem:energy} implies
\begin{align*}
 &\left(\energyintegral\right)^{1/2} \le C_n\left[
      \forcenorm+\headnorm^{3/4}+\forcenorm^{3/4}\right].
\end{align*}
Substitution gives
\begin{align*}
 \headnorm
 &\le C_n\forcenorm\headnorm^{1/2}
           +C_n\forcenorm\headnorm^{3/4}\\
 &\quad+C_n\left[
      \forcenorm+\forcenorm^{3/2}
                    +\forcenorm^2+\forcenorm^{7/4}\right].
\end{align*}
Both powers of the head norm on the right are strictly below one. Young's inequality gives, for every $\delta>0$,
\begin{align*}
 C_n\forcenorm\headnorm^{1/2}
 &\le\delta\headnorm+C_{n,\delta}\forcenorm^2,\\
 C_n\forcenorm\headnorm^{3/4}
 &\le\delta\headnorm+C_{n,\delta}\forcenorm^4.
\end{align*}
Take $\delta=1/4$ and absorb the terms $\headnorm$ to the left. Every remaining force power lies between $1$ and $4$, and is bounded by $\forcenorm+\forcenorm^4$. This proves \eqref{eq:head-final}.

For the final energy estimate, return to \eqref{eq:energy-inequality}. Since $\theta>2$, H\"older's inequality on the sphere and $\|u^r\|_{L^2(\Sph)}\le\|u\|_{L^2(\Sph)}$ give
\[
 \energyintegral
 \le C_n\bigl(\headnorm+\forcenorm\bigr)\|u\|_{L^2(\Sph)}.
\]
Insert the head bound \eqref{eq:head-final} and use the $L^2$ bound above. Young's inequality yields
\begin{align*}
 &\energyintegral\\
 &\quad\le C_n\bigl(\forcenorm+\forcenorm^4\bigr)
       \left(\energyintegral\right)^{1/2}\\
 &\quad\le\frac12\energyintegral
       +C_n\bigl(\forcenorm^2+\forcenorm^8\bigr).
\end{align*}
Absorbing half of the energy integral proves \eqref{eq:energy-final}.
\end{proof}

\section{Proof of the homogeneous theorem}\label{sec:existence}

We prove Theorem~\ref{thm:main} in the homogeneous setting, with $\theta=\frac{(n-2)(n-1)}{2(n-3)}$. We turn the estimates of Proposition~\ref{prop:apriori} into uniform classical bounds and then apply Leray--Schauder degree theory to a homogeneous Stokes operator. The construction is needed only for $n\ge17$; existence and small-force uniqueness for $4\le n\le16$ are established in \cite[Theorem~1.1]{B}.

We use the following local theorem of Li--Yang. It is stated for regular solutions and is valid in every dimension.

\begin{lemma}[Li--Yang, Proposition 2.3]\label{lem:local-regularity}
Let $n\ge2$, $\gamma>n/2$, and let $(u,p)$ be a regular solution of
\[
 -\Delta u+(u\cdot\nabla)u+\nabla p=g,
 \qquad\diver u=0\quad\text{in }B_1\subset\R^n,
\]
where $g\in L^\infty(B_1)$. Set $H=p+\frac12|u|^2$. If
\begin{align*}
 &\|u\|_{W^{1,2}(B_1)}+\|p\|_{W^{1,n/(n-1)}(B_1)} +\|g\|_{L^\infty(B_1)}
       +\|H_+\|_{L^\gamma(B_1)}\le C_0,
\end{align*}
then
\[
 \|u\|_{L^\infty(B_{1/2})}
       +\|\nabla u\|_{L^\infty(B_{1/2})}\le C,
\]
where $C$ depends only on $n,C_0$, and a positive lower bound for $\gamma-n/2$.
\end{lemma}

Lemma~\ref{lem:local-regularity} is quoted from \cite[Proposition~2.3, pp.~716--718]{LY}. The dimensional restriction in the existence theorem of \cite{LY} is not a restriction on this local result. Its regular-solution formulation suffices for the a priori argument below.

\begin{proposition}\label{prop:classical}
Let $n\ge17$. Every classical degree $-1$ homogeneous solution with force $f$ admits a degree $-2$ pressure and, for every $0<\alpha<1$, satisfies
\begin{equation}\label{eq:classical-bound}
 \|u\|_{C^{2,\alpha}(\Sph)}
 +\|p\|_{C^{1,\alpha}(\Sph)}
 \le C\bigl(n,\alpha,\|f\|_{\operatorname{Lip}(\Sph)}\bigr).
\end{equation}
The constant depends only on $n$, $\alpha$, and an upper bound for $\|f\|_{\Lip(\Sph)}$, and is independent of the particular solution.
\end{proposition}
\begin{proof}
We use the degree $-2$ pressure representative fixed by Lemma~\ref{lem:pressure-normalization}.

Set $N=(u\cdot\nabla)u-f$, a degree $-3$ vector field. Taking divergence of the equation gives $-\Delta p=\diver N$. Its spherical form is
\begin{equation}\label{eq:pressure-sphere}
 \bigl(-\Delta_{\Sph}+2(n-4)\bigr)p
       =\diver_{\Sph}N^t+(n-4)N^r.
\end{equation}
All fields in this identity are restricted to $\Sph$. The positive constant term and compact elliptic estimates give
\[
 \|p\|_{W^{1,\frac{n-1}{n-2}}(\Sph)}
       \le C_n\|N\|_{L^{\frac{n-1}{n-2}}(\Sph)}.
\]
Indeed, the divergence-data estimate in local charts has only an additional lower norm of $p$. That term is removed by compactness and the trivial kernel of $-\Delta_{\Sph}+2(n-4)$: a normalized sequence contradicting the estimate converges strongly in the lower norm to a homogeneous-equation solution, and testing that smooth limit by itself forces it to vanish. The divergence on the right costs only the displayed norm of $N$.

Sobolev embedding on $\Sph$, the energy integral, and H\"older's inequality imply
\begin{align*}
 \|u\|_{L^{\frac{2(n-1)}{n-3}}(\Sph)}
 &\le C_n\left(\energyintegral\right)^{\frac{1}{2}},\\
 \|\nabla u|_{\Sph}\|_{L^2(\Sph)}
 &\le\left(\energyintegral\right)^{\frac{1}{2}},\\
 \|N\|_{L^{\frac{n-1}{n-2}}(\Sph)}
 &\le C_n\left(\energyintegral+\forcenorm\right).
\end{align*}
The H\"older exponents satisfy
\[
 \frac{n-2}{n-1}=\frac{n-3}{2(n-1)}+\frac{1}{2}.
\]
Using the Sobolev conjugate $\frac{n-1}{n-3}$ and
$\nabla p|_{\Sph}=\nabla_{\Sph}p-2pe_r$, we obtain
\begin{equation}\label{eq:pressure-bound}
\begin{aligned}
 &\|p\|_{L^{\frac{n-1}{n-3}}(\Sph)}
       +\|\nabla p|_{\Sph}\|_{L^{\frac{n-1}{n-2}}(\Sph)}\\
 &\qquad\le C_n\left(\energyintegral+\forcenorm\right).
\end{aligned}
\end{equation}
This estimates the actual pressure, with its mean retained. In fact, integration of \eqref{eq:pressure-sphere} yields
\[
 2\int_{\Sph}p\dd\sigma=\int_{\Sph}N^r\dd\sigma.
\]

Fix $R>0$ small and finitely many centers $x_j\in\Sph$ such that $B_R(x_j)$ cover the sphere and $B_{2R}(x_j)$ remain in a fixed annulus away from zero. Homogeneity converts norms on $\Sph$ into annular norms with finite radial factors. Proposition~\ref{prop:apriori} and \eqref{eq:pressure-bound} therefore imply
\begin{align*}
 &\|u\|_{W^{1,2}(B_{2R}(x_j))}
 +\|p\|_{W^{1,\frac{n}{n-1}}(B_{2R}(x_j))}\\
 &\qquad+\|f\|_{L^\infty(B_{2R}(x_j))}
       +\|H_+\|_{L^\theta(B_{2R}(x_j))}
       \le C(n,\forcenorm).
\end{align*}
The pressure exponents suffice since
\[
 \frac{n-1}{n-3}>\frac{n-1}{n-2}>\frac{n}{n-1},
 \qquad \theta-\frac{n}{2}=\frac{1}{n-3}>0.
\]
Rescaling each doubled ball to $B_1$ by
\[
 \widetilde u(y)=2R\,u(x_j+2Ry),\qquad
 \widetilde p(y)=(2R)^2p(x_j+2Ry)
\]
gives force $(2R)^3f(x_j+2Ry)$ and head $(2R)^2H(x_j+2Ry)$. Since $R$ is fixed, all hypotheses of Lemma~\ref{lem:local-regularity} hold uniformly with $\gamma=\theta$. It follows that $u$ and $\nabla u$ are uniformly bounded on the sphere and, by homogeneity, on a fixed annulus $A = \{1/2 \le |x| \le 2 \}$.
 
We conclude by standard Stokes regularity. The interior $L^b$ estimates for the Stokes system with bounded right side $f-u\cdot\nabla u$ give uniform $W^{2,b}$ bounds for $u$ on smaller fixed annuli for every $1<b<\infty$; see \cite[Section IV.4]{Galdi}. Morrey's embedding \cite[Theorem 7.26]{GT}, with $b$ sufficiently large, then gives uniform $C^{1,\beta}$ bounds for some $\alpha<\beta<1$. Since $f$ is Lipschitz, $f-u\cdot\nabla u$ is uniformly bounded in $C^\alpha$. The interior Stokes Schauder estimates \cite[Section IV.7]{Galdi} give uniform $C^{2,\alpha}$ bounds for $u$ and $C^{1,\alpha}$ bounds for $p$, whose additive constant is controlled by the preceding pressure estimate. Restricting to $\Sph$ proves \eqref{eq:classical-bound}. Repeating the regularity argument gives smoothness away from zero when $f$ is smooth.
\end{proof}

\begin{lemma}\label{lem:stokes-unique}
Let $n>4$. A classical pair $(z,q)$ of degrees $-1,-2$, respectively, satisfying $-\Delta z+\nabla q=0$ and $\diver z=0$ on $\mathbb{R}^n\setminus\{0\}$ is zero.
\end{lemma}
\begin{proof}
Divergence gives $\Delta q=0$, so the degree $-2$ restriction satisfies $(-\Delta_{\Sph}+2(n-4))q=0$. Testing by $q$ on $\Sph$ gives $q=0$. Each component $z_j$ is then harmonic and degree $-1$, hence $(-\Delta_{\Sph}+n-3)z_j=0$ on the sphere. Testing by $z_j$ gives $z_j=0$. Elliptic regularity justifies the same argument if the equations initially hold distributionally.
\end{proof}

\begin{proof}[Proof of Theorem~\ref{thm:main}]
\emph{(i) Existence and regularity.} Fix $n\ge17$ and the prescribed force $f$. Following the degree theory argument of \cite[Section~4]{LY}, introduce the family
\begin{equation}\label{eq:homotopy}
 -\Delta u+(u\cdot\nabla)u+\nabla p=t f,
 \qquad\diver u=0,\qquad 0\le t\le1.
\end{equation}
Let $X$ be the Banach space of divergence-free degree $-1$ homogeneous $C^1$ fields on $\mathbb{R}^n\setminus\{0\}$, with norm
\[
 \|V\|_X=\|V|_{\Sph}\|_{L^\infty(\Sph)}
       +\|\nabla V|_{\Sph}\|_{L^\infty(\Sph)}.
\]
Restriction to the sphere identifies $X$ with a closed subspace of $C^1(\Sph;\R^n)$: radial differentiation is fixed by homogeneity and the divergence constraint is $\diver_{\Sph}V^t=-(n-2)V^r$. In particular the displayed norm is equivalent to the $C^1(\Sph)$ norm.

The Stokes velocity and pressure kernels are
\begin{align}
 \mathcal G_{ij}(x)
 &=\frac1{2(n-2)|\Sph|}
   \left(\frac{\delta_{ij}}{|x|^{n-2}}
             +(n-2)\frac{x_ix_j}{|x|^n}\right),\label{eq:Stokes-kernels}\\
 \Pi_j(x)&=\frac{x_j}{|\Sph|\,|x|^n}.
\end{align}
They satisfy $-\Delta\mathcal G_{ij}+\partial_i\Pi_j=\delta_{ij}\delta_0$ and $\partial_i\mathcal G_{ij}=0$ distributionally. For $0\le t\le1$ and $V\in X$, put
\[
 q_{t,V}=t f-(V\cdot\nabla)V
\]
and define
\begin{align*}
 T_i(t,V)(x)
 &=\int_{\R^n}\mathcal G_{ij}(x-y)q_{t,V,j}(y)\dd y,\\
 \pi_{t,V}(x)
 &=\int_{\R^n}\Pi_j(x-y)q_{t,V,j}(y)\dd y.
\end{align*}
Repeated indices are summed. $q_{t,V}$ is degree $-3$ and satisfies
\[
 |q_{t,V}(y)|
 \le C_n\bigl(t\forcenorm+\|V\|_X^2\bigr)|y|^{-3}.
\]
For fixed $x\ne0$, the integrals for $T$, $\nabla T$, and $\pi_{t,V}$ converge absolutely. Near $y=0$ the integrand is bounded by $|y|^{-3}$, which is integrable for $n>3$. Near $y=x$ the kernels have orders $|x-y|^{2-n}$, $|x-y|^{1-n}$, and $|x-y|^{1-n}$, while the data are bounded. At infinity they are bounded by $|y|^{-(n+1)}$, $|y|^{-(n+2)}$, and $|y|^{-(n+2)}$. 

The fundamental-solution identities and scaling give
\begin{equation}\label{eq:Stokes-map}
 -\Delta T+\nabla\pi_{t,V}
 =t f-(V\cdot\nabla)V,
 \qquad\diver T=0,
\end{equation}
and, for every $r>0$,
\[
 T(t,V)(rx)=r^{-1}T(t,V)(x),\qquad
 \pi_{t,V}(rx)=r^{-2}\pi_{t,V}(x).
\]
The integral estimates are uniform on the sphere and yield
\begin{equation}\label{eq:Stokes-map-bound}
 \|T(t,V)\|_X
       +\|\pi_{t,V}\|_{L^\infty(\Sph)}
 \le C_n\bigl(t\forcenorm+\|V\|_X^2\bigr).
\end{equation}
Thus $T$ maps into $X$.

On bounded subsets of $[0,1]\times X$, the right-hand side of \eqref{eq:Stokes-map} is uniformly bounded on the fixed annulus $A=\{1/2<|x|<2\}$. Set $A'=\{3/4<|x|<3/2\}\Subset A$. The standard interior Stokes estimates \cite[Section IV.4]{Galdi}, together with homogeneity, give, for every $1<b<\infty$,
\begin{equation}\label{eq:Stokes-map-compact}
\begin{aligned}
 &\|T(t,V)\|_{W^{2,b}(\Sph)}+\|\pi_{t,V}\|_{W^{1,b}(\Sph)}\\
 &\quad\le C_{n,b}\bigl(\|T(t,V)\|_{W^{2,b}(A')}+\|\pi_{t,V}\|_{W^{1,b}(A')}\bigr)\\
 &\quad\le C_{n,b}\bigl(\|t f-V\cdot\nabla V\|_{L^b(A)}
             +\|T(t,V)\|_{L^b(A)}+\|\pi_{t,V}\|_{L^b(A)}\bigr)\\
 &\quad\le C_{n,b}\bigl(t\forcenorm+\|V\|_X^2\bigr).
\end{aligned}
\end{equation}
The last inequality uses \eqref{eq:Stokes-map-bound} and homogeneity to control the lower norms. Choosing $b>n-1$ gives compact embedding into $C^1(\Sph)$, hence compactness in $X$.

For continuity, take $t_1,t_2\in[0,1]$ and $V_1,V_2\in X$, and use
\[
 (V_1\cdot\nabla)V_1-(V_2\cdot\nabla)V_2
 =((V_1-V_2)\cdot\nabla)V_1
       +(V_2\cdot\nabla)(V_1-V_2).
\]
The absolute kernel bounds give
\begin{equation}\label{eq:Stokes-map-continuity}
\begin{aligned}
 &\|T(t_1,V_1)-T(t_2,V_2)\|_X\\
 &\quad\le C_n|t_1-t_2|\forcenorm
       +C_n(\|V_1\|_X+\|V_2\|_X)\|V_1-V_2\|_X.
\end{aligned}
\end{equation}
Therefore $T:[0,1]\times X\to X$ is a continuous compact homotopy.

Every fixed point $u=T(t,u)$ is classical: interior Stokes regularity gives $u\in W^{2,b}_{\mathrm{loc}}$ for every finite $b>1$, and Sobolev embedding followed by the Stokes Schauder estimates gives $u\in C^{2,\alpha}_{\mathrm{loc}}$ and $\pi_{t,u}\in C^{1,\alpha}_{\mathrm{loc}}$ for every $0<\alpha<1$. These initial regularity bounds may depend on the fixed point. We can therefore apply the a priori estimates of Section~\ref{sec:apriori}.

For every $t\in[0,1]$, the force $t f$ is locally Lipschitz and homogeneous of degree $-3$. Proposition~\ref{prop:apriori} in Section~\ref{sec:apriori}, with $f$ replaced by $t f$, gives
\begin{equation}\label{eq:homotopy-head}
 \headnorm\le C_n\bigl(t\forcenorm+t^4\forcenorm^4\bigr)
\end{equation}
and
\begin{equation}\label{eq:homotopy-energy}
 \energyintegral
 \le C_n\bigl(t^2\forcenorm^2+t^8\forcenorm^8\bigr).
\end{equation}
These bounds are uniform for $0\le t\le1$. Moreover,
\[
 \|t f\|_{\Lip(\Sph)}=t\|f\|_{\Lip(\Sph)}
       \le\|f\|_{\Lip(\Sph)}.
\]
Proposition~\ref{prop:classical}, again with $f$ replaced by $t f$, therefore yields
\[
 \|u\|_{C^{2,\alpha}(\Sph)}+\|p\|_{C^{1,\alpha}(\Sph)}
 \le C\bigl(n,\alpha,\|f\|_{\Lip(\Sph)}\bigr),
\]
with a constant independent of $t$ and of the fixed point. In particular, all fixed points $u=T(t,u)$ are uniformly bounded in $X$.

For $t=0$, estimates \eqref{eq:homotopy-head} and \eqref{eq:homotopy-energy} have zero right sides. The energy integral vanishes, and its positive term $(n-3)\int_{\Sph}|u|^2\dd\sigma$ shows that $u=0$. Thus zero is the unique fixed point of $T(0,\cdot)$.

Moreover, the quadratic bound $\|T(0,V)\|_X\le C_n\|V\|_X^2$ in \eqref{eq:Stokes-map-bound} implies $D_VT(0,0)=0$. Thus $I-D_VT(0,0)$ is invertible, so the unique initial fixed point is nondegenerate. Leray--Schauder degree theory, applied to the compact homotopy $T$ and the uniform fixed-point bound, gives a fixed point of $T(1,\cdot)$; see \cite[Section 2.3 and Theorem 2.8.1]{Nirenberg}, as used in \cite[proof of Theorem 1.1]{LY}. This proves existence for $n\ge17$, with the required homogeneity and \eqref{eq:classical-bound}. Smoothness for smooth force follows by the same local regularity argument.

Existence and smoothness in $4\le n\le16$ follow from \cite[Theorem~1.1(i)--(ii)]{B}. This completes the existence and regularity assertions of Theorem~\ref{thm:main}.

\emph{(ii) Small-force uniqueness.} Let $n\ge17$ and take any classical homogeneous Navier--Stokes solution, with its normalized pressure. Subtraction of $(T(1,u),\pi_{1,u})$ produces a homogeneous Stokes pair to which Lemma~\ref{lem:stokes-unique} applies. Thus every such solution satisfies
\[
 u=T(1,u),\qquad p=\pi_{1,u}.
\]
The integral representation therefore imposes no smaller solution class.

All solutions with $\|f\|_{\operatorname{Lip}(\Sph)}\le1$ have a common $C^{2,1/2}(\Sph)$ velocity bound by Proposition~\ref{prop:classical}. Estimate~\eqref{eq:energy-final} gives
\[
 (n-3)\|u\|_{L^2(\Sph)}^2
 \le\energyintegral
 \le C_n\bigl(\forcenorm^2+\forcenorm^8\bigr).
\]
Their $X$ norms consequently tend uniformly to zero as the force Lipschitz norm tends to zero. Otherwise a sequence of such forces tending to zero and solutions bounded away from zero in $X$ would have a $C^1(\Sph)$ convergent subsequence, by the uniform classical bound. Its limit would have nonzero $X$ norm and zero $L^2$ norm, a contradiction.

For two solutions $u_1,u_2$ with the same force, \eqref{eq:Stokes-map-continuity} and their representations yield
\[
 \|u_1-u_2\|_X
 \le C_n(\|u_1\|_X+\|u_2\|_X)\|u_1-u_2\|_X.
\]
Choose the force threshold so that every solution has $X$ norm below $1/(4C_n)$. The coefficient is then less than one, and $u_1=u_2$. Two corresponding pressures have constant difference, which vanishes by degree $-2$ homogeneity. This proves uniqueness among all classical homogeneous solutions in the new range. Theorem~1.1(iii) of \cite{B} supplies $4\le n\le16$, completing Theorem~\ref{thm:main}.

\end{proof}

\section{The whole-space problem with compactly supported forces}\label{sec:whole-space}

This section gives the simplified proof of Theorem~\ref{thm:compact-force}. All fields are now defined on $\R^n$, and no homogeneity is assumed. The symbols $\theta$ and $q$ denote variable Lebesgue exponents subject to the hypotheses of each result; they are no longer fixed at the homogeneous values used in Sections~\ref{sec:apriori}--\ref{sec:existence}. We continue to write $H=p+\frac12|u|^2$ for the Bernoulli function, $f$ for the physical force, and $F$ for the auxiliary vector field. Here $F$ is $\R^n$-valued and is constructed without subtracting a spherical mean. Unless otherwise specified, all norms and integrals in this section are over $\R^n$. The two preliminary lemmas are stated under their own, more general assumptions.

The estimates parallel those of Section~\ref{sec:apriori}. The pressure calculation in Proposition~\ref{prop:compact-head} is the whole-space version of Lemma~\ref{lem:local-pressure}: decay permits removal of the outer cutoff, and we retain the weighted potential of $|u|^2$ directly, rather than converting it into the integral estimate of Lemma~\ref{lem:pressure-growth}. Lemma~\ref{lem:euclidean-kernel} corresponds to Lemma~\ref{lem:sphere-flux}; both estimate the gradient kernel by interpolation between its $L^1$ and $L^\infty$ bounds on dyadic annuli. On $\R^n$, the additional lower bound $q>n/(n-1)$ controls the part of the kernel at infinity.

Lemma~\ref{lem:euclidean-scalar} plays the role of the power test in Step~1 of Proposition~\ref{prop:head}. On the sphere, the chosen power cancels the tangential transport against the radial reaction term. In the whole space, $\diver u=0$ cancels the transport term for every admissible power, giving a family of finite-exponent estimates. Finally, Proposition~\ref{prop:compact-head} closes the estimates by Young's inequality, as in Proposition~\ref{prop:apriori}. The whole-space energy bound already depends only on $f$, whereas the preliminary homogeneous energy bound is coupled to $H_+$. This accounts for the simpler absorption in the whole-space proof.

\begin{lemma}\label{lem:euclidean-scalar}
Let $n>2$, let $\frac{2n}{n-2}\le\theta<\infty$, and define $q$ by $\frac1q=\frac1\theta+\frac1n$. Suppose $u\in C^1(\R^n;\R^n)$ satisfies $\diver u=0$ and $|u(x)|\le C_u(1+|x|)^{-1}$. Let $H\in W^{1,2}_{\mathrm{loc}}(\R^n)$ be locally bounded, with $H_+\in L^\theta(\R^n)$, and let $F\in L^q(\R^n;\R^n)$. If
\[
 -\Delta H+u\cdot\nabla H\le\diver F
\]
in distributions, then
\begin{equation}\label{eq:euclidean-scalar}
 \|H_+\|_{L^\theta}\le C_n\theta\|F\|_{L^q}.
\end{equation}
The constant is independent of $u$ and of the decay constant $C_u$.
\end{lemma}

\begin{proof}
Choose $\chi\in C_c^\infty(B_2)$ with $0\le\chi\le1$ and $\chi=1$ on $B_1$, and set $\chi_R(x)=\chi(x/R)$ for $R\ge1$. Thus $\chi_R=1$ on $B_R$, $\supp\chi_R\subset B_{2R}$, and $|\nabla\chi_R|\le C/R$. Put
\[
 \beta=\frac{(n-2)\theta}{n}\ge2,
 \qquad \frac2q+\frac{\beta-2}{\theta}=1,
 \qquad \frac{\beta n}{n-2}=\theta.
\]
H\"older's inequality gives
\[
 \int|F|^2H_+^{\beta-2}
 \le\|F\|_{L^q}^2\|H_+\|_{L^\theta}^{\beta-2}<\infty.
\]
When $\beta=2$, the integral on the left is read as $\int|F|^2$ and $q=2$.

We test the equation with the nonnegative function $\chi_R^2H_+^{\beta-1}$. $\beta\ge2$, and $q\ge2$. Using $\diver u=0$, the transport term becomes
\[
 \int\chi_R^2(u\cdot\nabla H)H_+^{\beta-1}
 =\frac1\beta\int\chi_R^2u\cdot\nabla(H_+^\beta)
 =-\frac2\beta\int\chi_R H_+^\beta u\cdot\nabla\chi_R.
\]
Integration by parts therefore yields
\begin{align*}
 (\beta-1)\int\chi_R^2H_+^{\beta-2}|\nabla H_+|^2
 &\le-(\beta-1)\int\chi_R^2H_+^{\beta-2}F\cdot\nabla H_+\\
 &\quad-2\int\chi_RH_+^{\beta-1}(F+\nabla H_+)\cdot\nabla\chi_R\\
 &\quad+\frac2\beta\int\chi_RH_+^\beta u\cdot\nabla\chi_R.
\end{align*}
Divide by $\beta-1$ and apply Young's inequality to the terms containing $F$ and $\nabla H_+$. Absorb the terms with $\nabla H_+$ to the left side, giving
\begin{align*}
 \int\chi_R^2H_+^{\beta-2}|\nabla H_+|^2
 &\le C\int\chi_R^2|F|^2H_+^{\beta-2}
       +C\int H_+^\beta|\nabla\chi_R|^2\\
 &\quad+C\int\chi_R|u|H_+^\beta|\nabla\chi_R|,
\end{align*}
where $C$ is a numerical constant independent of $R$, $\beta$, and $u$.

Both cutoff errors tend to zero. Indeed, $|u(x)|\le C_u/R$ on $B_{2R}\setminus B_R$, and H\"older's inequality with $\beta/\theta=(n-2)/n$ gives
\begin{align*}
 &\int H_+^\beta|\nabla\chi_R|^2
       +\int\chi_R|u|H_+^\beta|\nabla\chi_R|\\
 &\qquad\le\frac{C_n(1+C_u)}{R^2}
       \int_{B_{2R}\setminus B_R}H_+^\beta
 \le C_n(1+C_u)
       \left(\int_{B_{2R}\setminus B_R}H_+^\theta\right)^{\frac{n-2}{n}}
 \longrightarrow0.
\end{align*}
The last limit follows from $H_+\in L^\theta(\R^n)$. Letting $R\to\infty$ in the cutoff energy estimate, using Fatou's lemma on the left and dominated convergence on the first term on the right, yields
\begin{equation}\label{eq:euclidean-power-energy}
 \int H_+^{\beta-2}|\nabla H_+|^2
 \le C\int|F|^2H_+^{\beta-2}.
\end{equation}
In particular, $\nabla H_+^{\frac\beta2}\in L^2(\R^n)$.

We now apply Sobolev embedding to the compactly supported function $\chi_RH_+^{\frac\beta2}$:
\begin{align*}
 \|\chi_RH_+^{\frac\beta2}\|_{L^{\frac{2n}{n-2}}}
 &\le C_n\|\nabla(\chi_RH_+^{\frac\beta2})\|_{L^2}\\
 &\le C_n\|\chi_R\nabla H_+^{\frac\beta2}\|_{L^2}
       +C_n\|H_+^{\frac\beta2}\nabla\chi_R\|_{L^2}.
\end{align*}
The last term tends to zero by the cutoff estimate above. Since $H_+^{\frac\beta2}\in L^{\frac{2n}{n-2}}$ and its gradient belongs to $L^2$, dominated convergence in the other two terms gives, by \eqref{eq:euclidean-power-energy} and H\"older's inequality,
\begin{align*}
 \|H_+\|_{L^\theta}^{\frac\beta2}
 &=\|H_+^{\frac\beta2}\|_{L^{\frac{2n}{n-2}}}
 \le C_n\|\nabla H_+^{\frac\beta2}\|_{L^2}\\
 &\le C_n\beta\left(\int|F|^2H_+^{\beta-2}\right)^{\frac12}
 \le C_n\beta\|F\|_{L^q}\|H_+\|_{L^\theta}^{\frac{\beta-2}{2}}.
\end{align*}
If $\|H_+\|_{L^\theta}>0$, divide by its power $(\beta-2)/2$ and use $\beta\le\theta$ to obtain \eqref{eq:euclidean-scalar}. If $H_+=0$, the conclusion is immediate. The decay constant $C_u$ occurs only in terms that vanish as $R\to\infty$, so it does not enter the final estimate.
\end{proof}

\begin{samepage}
\begin{lemma}[A Euclidean analogue of Lemma~\ref{lem:sphere-flux}]\label{lem:euclidean-kernel}
Let $n>4$, $0<s<n-4$, and
\[
 \frac n{n-1}<q<\frac{n-2-s}{n-4-s}.
\]
If $g\in L^1(\R^n)$ and $\sup_x\int_{\R^n}|g(y)|^2|x-y|^{-s-2}\dd y<\infty$, then
\begin{equation}\label{eq:euclidean-kernel}
 \|\nabla(-\Delta)^{-1}g\|_{L^q}
 \le C_{n,s,q}\left[\|g\|_{L^1}
 +\left(\sup_x\int_{\R^n}\frac{|g(y)|^2}{|x-y|^{s+2}}\dd y\right)^{\frac12}\right].
\end{equation}
Here $(-\Delta)^{-1}$ is the Newtonian potential on $\R^n$. Compact support is not required in this lemma.
\end{lemma}
\end{samepage}

\begin{proof}
The proof follows the same near--far decomposition and annular interpolation as Lemma~\ref{lem:sphere-flux}. If $\Gamma$ is the fundamental solution of $-\Delta$, then $|\nabla\Gamma(z)|\le C_n|z|^{1-n}$. For $r_j=2^{-j}$, $j\ge0$, define
\[
 T_jg(x)=\int_{r_j/2<|x-y|\le r_j}|x-y|^{1-n}|g(y)|\dd y.
\]
Fubini-Tonelli's theorem gives
\[
 \|T_jg\|_{L^1}\le C_n r_j\|g\|_{L^1}.
\]
On the other hand, H\"older's inequality gives
\begin{align*}
 T_jg(x)
 &\le\left(\int_{r_j/2<|x-y|\le r_j}\frac{|g(y)|^2}{|x-y|^{s+2}}\dd y\right)^{\frac12}
       \left(\int_{r_j/2<|z|\le r_j}|z|^{s+4-2n}\dd z\right)^{\frac12}\\
 &\le C_{n,s}r_j^{-\frac{n-4-s}{2}}
       \left(\sup_z\int_{\R^n}\frac{|g(y)|^2}{|z-y|^{s+2}}\dd y\right)^{\frac12}.
\end{align*}
Interpolation between $L^1$ and $L^\infty$ therefore yields
\begin{align*}
 \|T_jg\|_{L^q}
 &\le C_{n,s,q}r_j^{\frac1q-\frac{n-4-s}{2}(1-\frac1q)}
       \|g\|_{L^1}^{\frac1q}
       \left(\sup_z\int_{\R^n}\frac{|g(y)|^2}{|z-y|^{s+2}}\dd y\right)^{\frac12(1-\frac1q)}.
\end{align*}
The power of $r_j$ is positive precisely when $q<(n-2-s)/(n-4-s)$. Summing over $j$ and applying weighted Young's inequality controls the part with $|x-y|\le1$ by the right side of \eqref{eq:euclidean-kernel}. For the remaining part $|x-y| > 1$, Young's convolution inequality gives
\[
 \|(\nabla\Gamma\,\mathbf1_{\{|z|>1\}})*g\|_{L^q}
 \le\|\nabla\Gamma\,\mathbf1_{\{|z|>1\}}\|_{L^q}\|g\|_{L^1}
 \le C_{n,q}\|g\|_{L^1},
\]
since $q>n/(n-1)$. These bounds also justify the gradient integral almost everywhere and complete the proof.
\end{proof}

\begin{proposition}[The compact-force Bernoulli estimate]\label{prop:compact-head}
Let $n\ge5$, let $f\in L^\infty(\R^n;\R^n)$ be supported in $B_{R_0}(0)$, and let $(u,p)$ be a regular solution of
\[
 -\Delta u+(u\cdot\nabla)u+\nabla p=f,\qquad\diver u=0\quad\text{in }\R^n.
\]
Suppose that, for a constant $C_u$ allowed to depend on this solution,
\begin{equation}\label{eq:compact-initial-decay}
 |u(x)|\le\frac{C_u}{(1+|x|)^{n-3}},\quad
 |\nabla u(x)|\le\frac{C_u}{(1+|x|)^{n-2}},\quad
 |p(x)|\le\frac{C_u\log(2+|x|)}{(1+|x|)^{n-1}}.
\end{equation}
Set $H=p+\frac12|u|^2$. For every finite $\theta\ge\frac{2n}{n-2}$ with $\theta>\frac n2$,
\begin{equation}\label{eq:compact-head-general}
 \|H_+\|_{L^\theta(\R^n)}\le C_{n,\theta,R_0,\|f\|_{L^\infty}},
\end{equation}
independently of $C_u$. In particular, $\|H_+\|_{L^n(\R^n)}\le C_{n,R_0,\|f\|_{L^\infty}}$.
\end{proposition}

\begin{proof}
Write $C_f$ for constants depending only on $n$, $R_0$, and $\|f\|_{L^\infty}$. The pressure in \eqref{eq:compact-initial-decay} is normalized to vanish at infinity. Testing the velocity equation with $u$, using Sobolev embedding, and then applying the pressure estimates of \cite[Lemma 2.2]{LY} gives
\begin{equation}\label{eq:compact-basic}
 \|\nabla u\|_{L^2}+\|u\|_{L^{\frac{2n}{n-2}}}
 +\|\nabla p\|_{L^{\frac n{n-1}}}+\|p\|_{L^{\frac n{n-2}}}
 +\|H\|_{L^{\frac n{n-2}}}\le C_f.
\end{equation}
For example, the energy estimate follows from
\[
 \int|\nabla u|^2=\int f\cdot u
 \le\|f\|_{L^{\frac{2n}{n+2}}}\|u\|_{L^{\frac{2n}{n-2}}}
 \le C_n\|f\|_{L^{\frac{2n}{n+2}}}\|\nabla u\|_{L^2}.
\]
The initial decay justifies these integrations. It also implies $H_+\in L^\theta$ for each exponent in the statement before any uniform bound is proved.

\emph{Step 1. The weighted potential of $|u|^2$.} For $0<s<n-4$ and any $y\in\R^n$, the pressure equation $\Delta p=-\partial_i\partial_j(u_i u_j)+\diver f$ gives
\begin{equation}\label{eq:compact-weighted-identity}
\begin{aligned}
 &(s+2)\int\frac{|u(x)\cdot(x-y)|^2}{|x-y|^{s+4}}\dd x
 +\frac{n-4-s}{2}\int\frac{|u(x)|^2}{|x-y|^{s+2}}\dd x\\
 &\qquad=(n-2-s)\int\frac{H(x)}{|x-y|^{s+2}}\dd x
 +\int\frac{f(x)\cdot(x-y)}{|x-y|^{s+2}}\dd x.
\end{aligned}
\end{equation}
Indeed, test against $|x-y|^{-s}$ and use
\[
 \partial_{ij}|x-y|^{-s}
 =-s\frac{\delta_{ij}}{|x-y|^{s+2}}
 +s(s+2)\frac{(x_i-y_i)(x_j-y_j)}{|x-y|^{s+4}}
\]
Substituting $p=H-\frac12|u|^2$ and dividing by $s$ gives \eqref{eq:compact-weighted-identity}. This is the unlocalized form of the computation in \cite[Lemma 2.10]{LY}. To justify the singular test, remove a ball of radius $\varepsilon$ about $y$ and insert an outer cutoff at radius $R$. Local boundedness makes the inner error $O(\varepsilon^{n-s-2})$. H\"older's inequality and \eqref{eq:compact-basic} bound the outer errors by $C_f R^{-s}$. Both vanish in the respective limits.

Fix
\[
 s=n-4-\frac1n.
\]
The coefficient $n-4-s$ is positive. Discard the first term on the left of \eqref{eq:compact-weighted-identity} and replace $H$ by $H_+$. On $|x-y|<1$, H\"older's inequality applies because
\[
 (s+2)\frac\theta{\theta-1}<n,
\]
which follows from $\theta>n/2$. On $|x-y|\ge1$, use \eqref{eq:compact-basic} and $(s+2)n/2>n$. Thus
\begin{align*}
 \int\frac{H_+(x)}{|x-y|^{s+2}}\dd x
 &\le C_{n,\theta}\|H_+\|_{L^\theta}
 +\|H_+\|_{L^{\frac n{n-2}}}
       \left(\int_{|x-y|\ge1}|x-y|^{-\frac{n(s+2)}2}\dd x\right)^{\frac2n}\\
 &\le C_{n,\theta}\|H_+\|_{L^\theta}+C_f.
\end{align*}
The force term is bounded uniformly in $y$ by $C_n\|f\|_{L^\infty}+\|f\|_{L^1}$, using $|f|\,|x-y|^{-s-1}$ and splitting at distance one. We obtain
\begin{equation}\label{eq:compact-weighted-bound}
 \sup_y\int\frac{|u(x)|^2}{|x-y|^{s+2}}\dd x
 \le C_{n,\theta}\|H_+\|_{L^\theta}+C_f.
\end{equation}
Keeping $s<n-4$ is essential here: the coefficient of the $|u|^2$ term in \eqref{eq:compact-weighted-identity} vanishes at the Newtonian exponent $s+2=n-2$.

\emph{Step 2. The auxiliary vector field and closure.} Define
\[
 g=|f||u|,\qquad F=-\nabla(-\Delta)^{-1}g-f,
 \qquad q=\frac{n\theta}{n+\theta}.
\]
Then $\diver F=g-\diver f$. The energy bound and \eqref{eq:compact-weighted-bound} give
\begin{align*}
 \|g\|_{L^1}
 &\le\|f\|_{L^{\frac{2n}{n+2}}}\|u\|_{L^{\frac{2n}{n-2}}}\le C_f,\\
 \sup_x\int\frac{|g(y)|^2}{|x-y|^{s+2}}\dd y
 &\le\|f\|_{L^\infty}^2\sup_x\int\frac{|u(y)|^2}{|x-y|^{s+2}}\dd y
 \le C_{n,\theta,f}\bigl(1+\|H_+\|_{L^\theta}\bigr).
\end{align*}
The exponents satisfy
\[
 \frac n{n-1}<2\le q<n<2n+1=\frac{n-2-s}{n-4-s},
\]
so Lemma~\ref{lem:euclidean-kernel}, together with $f\in L^q$, yields
\begin{equation}\label{eq:compact-F}
 \|F\|_{L^q}\le C_{n,\theta,f}\bigl(1+\|H_+\|_{L^\theta}^{\frac12}\bigr).
\end{equation}
Here and below $C_{n,\theta,f}$ also includes the fixed support radius $R_0$.

The Bernoulli identity gives, in distributions,
\[
 -\Delta H+u\cdot\nabla H
 =-\frac12\sum_{i,j}(\partial_i u_j-\partial_j u_i)^2
     +f\cdot u-\diver f
 \le\diver F.
\]
All hypotheses of Lemma~\ref{lem:euclidean-scalar} hold by the initial regularity, \eqref{eq:compact-initial-decay}, and \eqref{eq:compact-F}. Combining the two lemmas and applying Young's inequality gives
\[
 \|H_+\|_{L^\theta}
 \le C_{n,\theta}\|F\|_{L^q}
 \le C_{n,\theta,f}\bigl(1+\|H_+\|_{L^\theta}^{\frac12}\bigr)
 \le\frac12\|H_+\|_{L^\theta}+C_{n,\theta,f}.
\]
This proves \eqref{eq:compact-head-general}. Taking $\theta=n$ gives the stated $L^n$ bound, sufficient for the local regularity criterion of \cite[Proposition 2.3]{LY}.
\end{proof}

\begin{proof}[Proof of Theorem~\ref{thm:compact-force}]
Fix $n\ge5$ and the prescribed compactly supported force $f$. We use the weighted space and Stokes construction of \cite{LY}. Specifically, let $C_d^1(\R^n)$ be the closure of the smooth compactly supported divergence-free fields under
\[
 \|V\|_{C_d^1}
 =\sup_{x\in\R^n}(1+|x|)^{n-3}|V(x)|
 +\sup_{x\in\R^n}(1+|x|)^{n-2}|\nabla V(x)|.
\]
With the Stokes kernels in \eqref{eq:Stokes-kernels}, define
\[
 T_i(t,V)(x)=\int_{\R^n}\mathcal G_{ij}(x-y)
       \bigl(tf_j(y)-V_k(y)\partial_kV_j(y)\bigr)\dd y,
 \qquad 0\le t\le1.
\]
The associated pressure is given by replacing $\mathcal G_{ij}$ with $\Pi_j$. The kernel estimates and local Stokes regularity in \cite[Lemma 4.1]{LY} show that $T:[0,1]\times C_d^1\to C_d^1$ is continuous and compact for every $n\ge5$. In particular, the argument only uses $2n-5\ge n$, with a harmless logarithm in the borderline case $n=5$; it imposes no upper bound on $n$.

Every fixed point $u=T(t,u)$ is regular, with $u\in W^{2,b}_{\mathrm{loc}}$ and its associated pressure in $W^{1,b}_{\mathrm{loc}}$ for all finite $b>1$. It satisfies \eqref{eq:compact-initial-decay}; see \cite[(2.1)--(2.2)]{LY}. The constants in those initial bounds may depend on the fixed point. Apply Proposition~\ref{prop:compact-head} with $f$ replaced by $tf$ and choose $\theta=n$. Since $\|tf\|_{L^\infty}\le\|f\|_{L^\infty}$ and $\supp(tf)\subset B_{R_0}$, it gives, uniformly for all fixed points and $0\le t\le1$,
\[
 \|H_+\|_{L^n(\R^n)}\le C_f.
\]
The basic bounds \eqref{eq:compact-basic}, again with force $tf$, and Lemma~\ref{lem:local-regularity} on translated unit balls therefore imply
\[
 \|u\|_{L^\infty(\R^n)}+\|\nabla u\|_{L^\infty(\R^n)}\le C_f.
\]
The local $W^{1,2}$ norm of $u$ and the local $W^{1,\frac n{n-1}}$ norm of $p$ required there follow from \eqref{eq:compact-basic} by H\"older's inequality on unit balls.

We can now apply the proof of \cite[Theorem 1.5, Section 3]{LY}. Its energy estimate outside a large ball, the vorticity regularity criterion, and the Stokes estimates give the uniform bound
\[
 \|u\|_{C_d^1(\R^n)}\le C_f.
\]
Those arguments hold for every $n\ge5$. The restriction $n\le15$ in their Bernoulli estimate is explicitly identified in \cite[Remark 2.12]{LY}; Proposition~\ref{prop:compact-head} replaces that estimate.

At $t=0$, testing the equation with $u$ gives $\int|\nabla u|^2=0$, so decay implies $u=0$. Also $T(0,V)$ is quadratic in $V$, whence $D_VT(0,0)=0$. The uniform fixed-point bound and Leray--Schauder degree theory therefore yield a fixed point of $T(1,\cdot)$, as in \cite[Section 4]{LY}; see \cite[Section 2.3 and Theorem 2.8.1]{Nirenberg}. The final kernel estimates in the proof of \cite[Theorem 1.1]{LY} improve the initial weighted decay to
\[
 |u(x)|\le\frac{C_f}{(1+|x|)^{n-2}},\qquad
 |\nabla u(x)|+|p(x)|\le\frac{C_f}{(1+|x|)^{n-1}}.
\]
Standard interior Stokes regularity gives the stated Sobolev regularity and smoothness for smooth $f$. This proves Theorem~\ref{thm:compact-force}.
\end{proof}

\section*{Acknowledgments}
I thank my thesis advisor, Professor Yanyan Li, for introducing me to this question.

\section*{Declaration on the use of artificial intelligence}
The author used ChatGPT during the preparation of this paper. The author has checked all mathematical arguments and details by hand and takes full responsibility for the content and accuracy of the paper.

\end{document}